\documentclass{amsart}

\usepackage{amsmath, amsthm}
\usepackage{graphics,setspace,xypic}
\usepackage{eucal}

\usepackage{amssymb}
\usepackage{amscd}
\usepackage{palatino}
\usepackage{latexsym}
\usepackage{epsfig}
\usepackage{graphicx}
\usepackage{amsfonts}
\usepackage{psfrag}
\usepackage{youngtab, ytableau}
\usepackage{url}
\usepackage{cite}

\usepackage[margin=1in]{geometry}

\input xy
\xyoption{all}
\UseComputerModernTips

\numberwithin{equation}{section}
\numberwithin{figure}{section}

\newtheorem{theorem}{Theorem}[section]
\newtheorem{lemma}[theorem]{Lemma}
\newtheorem{proposition}[theorem]{Proposition}
\newtheorem{corollary}[theorem]{Corollary}

\newtheorem{conjecture}[theorem]{Conjecture}
\newtheorem{remark}[theorem]{Remark}
\newtheorem{example}[theorem]{Example}

\theoremstyle{definition}
\newtheorem{definition}[theorem]{Definition}

\newcommand{\C}{{\mathbb{C}}}

\newcommand{\Z}{{\mathbb{Z}}}

\DeclareMathOperator{\Stab}{Stab}

\DeclareMathOperator{\Des}{Des}

\newcommand{\hsm}{{\hspace{1mm}}}

\usepackage{multicol}
\usepackage{color}
\definecolor{grey}{rgb}{0.7, 0.75, 0.71}

\newcommand{\W}{\mathcal{W}}
\newcommand{\A}{\mathcal{A}}

\newcommand{\Par}{{\mathrm{Par}}}

\newcommand{\Symm}{{\mathfrak{S}}}
\renewcommand{\Stab}{{\mathrm{Stab}}}
\newcommand{\Flags}{{\mathcal{F}\ell ags}}

\newcommand{\Hess}{{\mathcal{H}ess}}
\newcommand{\src}{{\mathrm{src}}}
\newcommand{\tgt}{{\mathrm{tgt}}}
\newcommand{\sk}{{\mathrm{sk}}}
\newcommand{\SK}{{\mathrm{SK}}} 
 
\newcommand{\asc}{{\mathrm{asc}}}
\newcommand{\inv}{{\mathrm{inv}}}

\begin{document}

\title{Upper-triangular linear relations on multiplicities and the Stanley--Stembridge conjecture}

\author{Megumi Harada}
\address{Department of Mathematics and
Statistics\\ McMaster University\\ 1280 Main Street West\\ Hamilton, Ontario L8S4K1\\ Canada}
\email{Megumi.Harada@math.mcmaster.ca}
\urladdr{\url{http://www.math.mcmaster.ca/Megumi.Harada/}}
\thanks{The first author is partially supported by a Canada Research Chair (Tier 2) Award and an NSERC Discovery Grant.}

\author{Martha Precup}
\address{Department of Mathematics and Statistics\\ Washington University in St. Louis \\ One Brookings Drive \\ St. Louis, Missouri  63130 \\ U.S.A. }
\email{martha.precup@wustl.edu}
\urladdr{\url{https://www.math.wustl.edu/~precup/}}


\date{\today}

%%%%%%%%%%%%%%%%%%%%%
%  Abstract
%%%%%%%%%%%%%%%%%%%%%

\begin{abstract}
In 2015, Brosnan and Chow, and independently Guay-Paquet, proved the Shareshian--Wachs conjecture, which links the Stanley--Stembridge conjecture in combinatorics to the geometry of Hessenberg varieties through Tymoczko's permutation group action on the cohomology ring of regular semisimple Hessenberg varieties. In previous work, the authors exploited this connection to prove a graded version of the Stanley--Stembridge conjecture in a special case. In this manuscript, we derive a new set of linear relations satisfied by the multiplicities of certain permutation representations in Tymoczko's representation. We also show that these relations are upper-triangular in an appropriate sense, and in particular, they uniquely determine the multiplicities. As an application of these results, we prove an inductive formula for the multiplicity coefficients corresponding to partitions with a maximal number of parts. 
%It follows from our formula that these coefficients are non-negative, thus giving additional positive evidence for the graded Stanley--Stembridge conjecture in the general case. 
\end{abstract}

\maketitle

\setcounter{tocdepth}{1}
\tableofcontents

%%%%%%%%%%%%%%%%%%%%%
%  Introduction
%%%%%%%%%%%%%%%%%%%%%

\section{Introduction}\label{sec:intro}

Recent results have forged exciting new connections between algebraic combinatorics and the geometry and topology of certain subvarieties of the flag variety called \textit{Hessenberg varieties}.  In particular, the Shareshian--Wachs conjecture \cite{ShareshianWachs2016}, proven in 2015 by Brosnan and Chow \cite{BrosnanChow2015} (and independently by Guay-Paquet \cite{Guay-Paquet2016}), established a new connection between Hessenberg varieties and the long-standing \textit{Stanley--Stembridge conjecture} in combinatorics, which states that the chromatic symmetric function of the incomparability graph of a (3+1)-free poset is e-positive, i.e., it is a non-negative linear combination of elementary symmetric functions. This is a well-known conjecture in the field of combinatorics which is related to various other deep conjectures about immanants. 
These recent results have established the following research problem: \emph{use the properties of Hessenberg varieties to prove the Stanley--Stembridge conjecture.} The problem can in fact be made more specific, as follows. 
The results of Brosnan--Chow and Guay-Paquet connect the \textit{dot action representation}, defined by Tymoczko in \cite{Tym08} on the cohomology groups of regular semisimple Hessenberg varieties, to the Stanley--Stembridge conjecture.  From this it follows that if Tymoczko's dot action representation is a permutation representation in which each point stabilizer is a Young subgroup, then the Stanley--Stembridge conjecture is true. We refer the reader to \cite[Introduction and Section 2]{HaradaPrecup2017} for a more leisurely account of the historical background and motivation for this circle of ideas.

There are already substantive partial results to the problem stated above. Most recently, we used Hessenberg varieties to prove a graded refinement of the Stanley--Stembridge conjecture in the so-called \textit{abelian case} by giving an inductive description of the nontrivial permutation representations that appear in that case.\cite{HaradaPrecup2017}. Moreover, in that manuscript we additionally stated a conjecture which gives, in the general case, an inductive description of the multiplicities of certain nontrivial permutation representations \cite[Conjecture 8.1]{HaradaPrecup2017}. 
Our main motivation for the present manuscript was to prove this conjecture using the geometry and combinatorics of Hessenberg varieties. In doing so, we discovered new properties obeyed by the multiplicities of the so-called tabloid representations in Tymoczko's representation, as we now explain.

We now describe in more detail the results of this manuscript. Hessenberg varieties in type A are subvarieties of the full flag variety $\Flags(\C^n)$ of nested sequences of linear subspaces in $\C^n$.  These varieties are parameterized by a choice of linear operator $\mathsf{X}\in \mathfrak{gl}(n,\C)$ and Hessenberg function $h: [n]=\{1,2,\ldots,n\} \to [n]=\{1,2,\ldots,n\}$.  (For details see Section~\ref{sec:background}.) For the purpose of this discussion it suffices to consider only the case when the operator is a regular semisimple operator $\mathsf{S}$ in $\mathfrak{gl}(n, \C)$; we denote the corresponding Hessenberg variety by $\Hess(\mathsf{S}, h)$. As mentioned above, Tymoczko defined \cite{Tym08} an action of the symmetric group $\Symm_n$ on $H^{2i}(\Hess(\mathsf{S},h))$ for each $i \geq 0$. From the work of Shareshian--Wachs, Brosnan--Chow, and Guay-Paquet it follows that in order to prove the (graded) Stanley--Stembridge conjecture, it suffices to 
prove that the cohomology 
$H^{2i}(\Hess(\mathsf{S},h))$ for each $i$ is a non-negative combination of the tabloid representations $M^\mu$ \cite[Part II, Section 7.2]{Ful97} of $\Symm_n$ for $\mu$ a partition of $n$. In other words, given the decomposition 
\begin{equation}\label{eq: decomp intro}
H^{2i}(\Hess(\mathsf{S},h)) = \sum_{\mu \vdash n}  c_{\mu, i} M^\mu
\end{equation}
in the representation ring $\mathcal{R}ep(\Symm_n)$ of $\Symm_n$, it suffices to show that the coefficients $c_{\mu,i}$ are non-negative.

We take a moment to mention here that the coefficients $c_{\mu,i}$ appearing in~\eqref{eq: decomp intro} were previously known to satisfy a matrix equation 
\[
\sum_{\mu \vdash n}  N_{\lambda \mu} c_{\mu, i} = y_{\lambda, i} 
\]
where the $y_{\lambda,i}$ are derived from Betti numbers of certain regular Hessenberg varieties and $N_{\lambda \mu} = \sum_{\nu \vdash n} K_{\nu, \lambda} K_{\nu, \mu}$ where the $K_{\nu, \lambda}, K_{\nu, \mu}$ are the Kostka numbers \cite[Section 2]{HaradaPrecup2017}. However, the Kostka numbers and the matrix $N$ are well-known to be computationally unwieldy, and it was not clear (to us) how to exploit the above matrix equation to prove the non-negativity of the $c_{\mu, i}$.  Another motivation for this manuscript was to find other relations satisfied by these coefficients which are more computationally tractable.

The main results of this manuscript are as follows. 
Let $n$ be a positive integer and $h: [n] \to [n]$ a Hessenberg function. Let $i\geq 0$ be a fixed non-negative integer and   $X_i = (c_{\mu,i})$ denote the (column) vector whose entries are the coefficients appearing in~\eqref{eq: decomp intro} above. 

\begin{itemize} 
\item In Corollary~\ref{corollary: graded matrix equation}, we derive a family of (new) matrix equations $A X_i = W_i$ satisfied by the column vectors $X_i$ for $i \geq 0$. The matrix $A = (A(\lambda, \mu))_{\lambda, \mu \vdash n}$ is obtained by counting certain subsets of the permutation group $\Symm_n$ using the data of a pair of partitions $\lambda, \mu \vdash n$, and is independent of both the choice of Hessenberg function $h$ and the integer $i \geq 0$. The column vectors $W_i$ are obtained by counting certain subsets of the permutation group $\Symm_n$ using the data of a partition $\lambda$, the Hessenberg function $h$, and the integer $i \geq 0$. 
\item In Theorem~\ref{theorem: upper-triangular}, we prove that the above matrix $A = (A(\lambda, \mu))$ is \textit{upper-triangular, with $1$'s along the diagonal}, with respect to an appropriately chosen linear order on the set $\Par(n)$ of partitions of $n$. We additionally prove an inductive formula for its matrix entries (Proposition~\ref{proposition: M-formulas Martha version}, cf. also Corollary~\ref{corollary: truncate A}). 
\item Generalizing results of \cite[Section 4]{HaradaPrecup2017}, we obtain a \text{sink set decomposition} of the subsets of $\Symm_n$ defining the column vector $W_i$ above (Proposition~\ref{prop: prop1}). As a consequence we obtain an inductive formula for the entries of $W_i$ for the special case in which $\lambda$ has the maximal possible number of parts (Theorem~\ref{thm: inductive formula}). 
\item As an application of the above results, we prove \cite[Conjecture 8.1]{HaradaPrecup2017}; more precisely, we obtain an inductive formula for the coefficients $c_{\mu, i}$ in~\eqref{eq: decomp intro} for the special case in which $\mu$ has the maximal possible number of parts (Theorem~\ref{theorem: max sink set case}), thus providing further evidence for the Stanley--Stembridge conjecture. 
\end{itemize} 

Some remarks are in order. Firstly, the main contribution of this manuscript are the 
new linear relations  in Corollary~\ref{corollary: graded matrix equation}; most particularly, the upper-triangularity of the matrix $A$ gives substantial reason to expect that these matrix equations will play a significant role in the solution to the full Stanley--Stembridge conjecture. Secondly, we are aware that there exist other proofs of our conjecture as stated in \cite[Conjecture 8.1]{HaradaPrecup2017}, using the coproduct structure on the ring of symmetric functions \cite{Lee-personal}. Thirdly, in his original paper on the subject, Stanley derives a different set of linear relations obeyed by the coefficients $c_\lambda$ \cite[Theorem 3.4, cf. also the erratum posted on Stanley's personal webpage]{Stanley1995, Stanley-erratum}, in which he uses a notion of \textit{sink sequences} which appear to be related to our sink-set decompositions. 

We now give a brief overview of the contents of the manuscript. Section~\ref{sec:background} is devoted to the setup and definitions of appropriate notation and terminology. In Section~\ref{sec:relations} we derive the new matrix equations $AX_i=W_i$, and in Section~\ref{sec:upper triangular} we prove that $A$ is upper-triangular, with $1$'s along the diagonal. We also derive the inductive formula for the numbers $A(\lambda, \mu)$. In Section~\ref{sec: inductive formula} we derive a separate inductive formula for the entries of the ``constant vector'' $W_i$. Finally, in Section~\ref{sec:closed formula} we prove Conjecture 8.1 from \cite{HaradaPrecup2017}.

\bigskip

\noindent \textbf{Acknowledgements.}  We are grateful for the hospitality and financial support of the Fields Institute for Research in the Mathematical Sciences in Toronto, Canada.  The Fields Research Fellowship allowed us to spend a fruitful month together at the Fields Institute  in August 2018, during which we had many of the ideas in this manuscript.

%%%%%%%%%%%%%%%%%%%%%
%  Background
%%%%%%%%%%%%%%%%%%%%%

\section{Background and Terminology}\label{sec:background}

In this section we briefly recall the setting of our paper. For a more leisurely account we refer to \cite{HaradaPrecup2017}. 
Hessenberg varieties in Lie type A are subvarieties of the (full) flag variety
$\Flags(\C^n)$, which is the collection of sequences of nested linear subspaces of $\C^n$:
\[
\Flags(\C^n) := 
\{ V_{\bullet} = (\{0\} \subset  V_1 \subset  V_2 \subset  \cdots V_{n-1} \subset 
V_n = \C^n) \hsm \vert \hsm \dim_{\C}(V_i) = i \ \textrm{for all} \ i=1,\ldots,n\}. 
\]
A Hessenberg variety in $\Flags(\C^n)$ is specified by two pieces of data: a \textbf{Hessenberg function}, that is, a nondecreasing function $h:\{1,2,\ldots,n\} \rightarrow \{1,2,\ldots,n\}$ such that $h(i) \geq i$ for all $i$, and a choice of an element $\mathsf{X}$ in $\mathfrak{gl} (n,\C)$. We frequently write a Hessenberg function by listing its values in sequence,
i.e., $h = (h(1), h(2), \ldots, h(n))$. 
The \textbf{Hessenberg variety} associated to the linear operator $\mathsf{X}$ and Hessenberg function $h$ and is defined as
\begin{equation}\label{eq: definition Hess X h}
{\mathcal{H}ess}(\mathsf{X},h) = \{ V_{\bullet}\in \Flags(\C^n) \hsm\vert\hsm \mathsf{X}V_i \subseteq V_{h(i)} \textup{   for all   } i\}.
\end{equation}

When the linear operator $\mathsf{X}$ is chosen to be a regular semisimple operator $\mathsf{S}$ (i.e., diagonalizable with distinct eigenvalues), we refer to the corresponding Hessenberg variety $\Hess(\mathsf{S}, h)$ as a \textbf{regular semisimple Hessenberg variety}. Tymoczko defined an action of the symmetric group $\Symm_n$ on the cohomology of a regular semisimple Hessenberg variety $H^*(\Hess(\mathsf{S},h))$ which is called the \textbf{dot action} \cite{Tym08}.  This action preserves the grading on cohomology, so in fact $\Symm_n$ acts on each $H^{2i}(\Hess(\mathsf{S},h))$ for $i \geq 0$ (the cohomology is concentrated in even degrees).  For $\mu$ a partition of $n$, we denote by $M^\mu$ the complex vector space with basis given by the set of tabloids of shape $\mu$. Since $\Symm_n$ acts on the set of tabloids, $M^\mu$ is a $\Symm_n$-representation, and is called the tabloid representation (corresponding to $\mu$) \cite[Part II, Section 7.2]{Ful97}. 
It is well-known that the set of these tabloid representations form 
a $\Z$-basis for the representation ring $\mathcal{R}ep(\Symm_n)$ of $\Symm_n$, so we can decompose $H^*(\Hess(\mathsf{S},h))$ with respect to Tymoczko's dot action as follows: 
\begin{equation}\label{eq: decomp into Mlambda}
H^*(\Hess(\mathsf{S},h)) = \sum_{\mu\vdash n} c_\mu M^\mu \quad \textup{ and } 
\quad H^{2i}(\Hess(\mathsf{S},h)) = \sum_{\mu \vdash n} c_{\mu,i} M^{\mu} 
\end{equation}
where $c_\mu, c_{\mu,i} \in \Z$.

As explained in the Introduction, the motivation of this manuscript is to prove the graded Stanley--Stembridge conjecture. We refer the reader to \cite{HaradaPrecup2017} for more history; for the present manuscript we take the `graded Stanley--Stembridge conjecture' to mean the following. 

\begin{conjecture}\label{conj:Stanley--Stembridge} 
Let $n$ be a positive integer, $h: [n] \to [n]$ be a Hessenberg function, and $\mathsf{S}$ be a regular semisimple linear operator. Then 
the integers $c_{\mu,i}$ appearing in~\eqref{eq: decomp into Mlambda} are non-negative. 
\end{conjecture}

\subsection{Hessenberg data} For later use, we introduce some Lie-theoretic and combinatorial notation associated to Hessenberg varieties.  We fix a Hessenberg function $h: [n] \to [n]$.

Let $\mathfrak{t} \subseteq \mathfrak{gl} (n,\C)$ denote the Cartan subalgebra of diagonal matrices and let $t_i$ denote the coordinate on $\mathfrak{t}$ reading off the $(i,i)$-th matrix entry along the diagonal. Denote the root system of $\mathfrak{gl}(n,\C)$ by $\Phi$. Then the positive roots of $\mathfrak{gl} (n,\C)$ are $\Phi^+ = \{ t_i - t_j \hsm \vert \hsm 1 \leq i < j \leq n\}$ where {$t_i - t_j \in \Phi^+$} corresponds to the root space spanned by the elementary matrix $E_{ij}$, denoted { $\mathfrak{g}_{t_i - t_j}$}. Similarly, the negative roots of $\mathfrak{gl} (n,\C)$ are $\Phi^- = \{ t_i - t_j \hsm \vert \hsm 1 \leq j < i \leq n\}$. We denote the simple positive roots in $\Phi^+$ by $\Delta = \{\alpha_i := t_i - t_{i+1} \; \vert \; 1 \leq i \leq n-1\}$.  {Finally, it} is clear that each root $t_i-t_j\in \Phi$ can be uniquely identified with an ordered pair $(i,j)$, with $i \neq j$. We will make this identification below whenever it is notationally convenient.

For each permutation $w\in \Symm_n$, let
\[
\inv(w): = \{ (i,j) \hsm\vert\hsm i>j \textup{ and } w(i)<w(j) \}
\]
denote the set of inversions of $w$. Note that we adopt the nonstandard notation of listing the larger number in the pair $(i,j)\in \inv(w)$ first.  This is because we frequently identity $\inv(w)$ with a subset of negative roots. Under the correspondence between ordered pairs and roots discussed in the last paragraph, this set indexes the negative roots which become positive under the action of $w$.   This action can be expressed concretely as $w(t_i-t_j) = t_{w(i)}-t_{w(j)}$.  

The Hessenberg function $h: [n]\to [n]$ uniquely determines two subsets of roots  as follows:
\[
\Phi_h^-:= \{t_i-t_j \hsm\vert\hsm i>j \textup{ and } i\leq h(j)\} \;\textup{ and }\; \Phi_h:= \Phi_h^- \sqcup \Phi^+.
\]
Let $\inv_h(w) := \inv(w)\cap \Phi_h^-$; this set of inversions is used later to compute the Betti numbers of certain Hessenberg varieties.

Recall that an ideal $I$ of $\Phi^-$ is defined to be a collection of negative roots such that if $\alpha\in I$, $\beta\in \Phi^-$, and $\alpha+\beta\in \Phi^-$, then $\alpha+\beta\in I$.  The relation defining $\Phi_h^-$ immediately implies that
\[
I_h:= \Phi^- \setminus \Phi_h^-
\]
is an ideal in $\Phi^-$.  We call it the \textbf{ideal corresponding to $h$}.   

Given an ideal $I\subseteq \Phi^-$, its lower central series is the sequence of ideals defined inductively by
\[
{I}_1={I} \;\textup{ and } \; {I}_j =\{ \gamma+\beta \hsm\vert\hsm \gamma,\,\beta\in I_{j-1} \textup{ and } \gamma+\beta \in \Phi^- \} \; \textup{ for all $j\geq 2$}.
\]
The \textbf{height of an ideal $I$} is the length of its lower central series and we denote it by $ht(I)$.

\begin{example}\label{sec2-ex1} Let $h=(2,4,4,5,5)$.  Then 
\[
\Phi_h^- = \{ t_2-t_1, t_3-t_2, t_4-t_2, t_4-t_3, t_5-t_4 \} \; \textup{ and } \; I_h = \{ t_3-t_1, t_4-t_1, t_5-t_1, t_5-t_2, t_5-t_3  \}
\]
and $ht(I_h)=2$ since
\[
(I_h)_2 = \{ t_5-t_1 \} \; \textup{ and } \; (I_h)_3=\emptyset.
\]
\end{example}

The data of a Hessenberg function can also be encoded by way of a graph.  Given a Hessenberg function $h: [n]\to [n]$,  the \textbf{incomparability graph} associated to $h$ is the graph $\Gamma_h =(V_h, E_h)$ with vertex set $V_h=[n]$ and edge set $E_h=\{\{i,j\} \,\vert\, i<j \textup{ and } h(i)\geq j \}$.  Notice that the edges of $\Gamma_h$ correspond bijectively to the roots in $\Phi_h^-$.

\begin{example}\label{sec2-ex2} The graph corresponding to the Hessenberg function $h=(2,4,4,5,5)$ from Example~\ref{sec2-ex1} is 
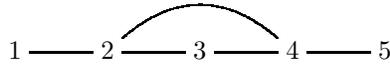
\vspace*{.15in}
\[
\xymatrix{1 \ar@{-}[r]   & 2 & 3 \ar@{-}[l]\ar@{-}[r]   & 4 \ar@{-}@/_1.5pc/[ll]  \ar@{-}[r]  & 5}
\]
\end{example}

In many ways, the combinatorial structure of the graph $\Gamma_h$ and  the ideal $I_h$ mirror one another.  For example, \cite[Proposition 5.8]{HaradaPrecup2017} shows  that $m(\Gamma_h)=ht(I_h)+1$, where $m(\Gamma_h)$ denotes the maximum cardinality of an independent subset of vertices (that is, vertices which are pairwise nonadjacent) in $\Gamma_h$.  The reader can confirm this equation for the Hessenberg function $h=(2,4,4,5,5)$ appearing in Example~\ref{sec2-ex1} and Example~\ref{sec2-ex2}.  This correspondence is essential for the arguments of Section~\ref{sec: inductive formula} below. Furthermore, the structure of the ideal $I_h$, and that of the graph $\Gamma_h$, is closely connected  to the dot action representation.  The following theorem relates the multiplicities of the tabloid representations appearing in~\eqref{eq: decomp into Mlambda} with the height of $I_h$.  This is a restatement of \cite[Corollary 5.12]{HaradaPrecup2017}.

\begin{theorem} \label{lemma: vanishing coefficients} Let $c_{\mu}$ and $c_{\mu, i}$ be the coefficients appearing in~\eqref{eq: decomp into Mlambda}.  Then $c_{\mu}=c_{\mu,i}=0$ for all $\mu\vdash n$ with more than $m(\Gamma_h) = ht(I_h)+1$ parts.
\end{theorem}

\subsection{Partitions and subsets of simple positive roots}\label{subsec: partitions}

In this section we establish some combinatorial terminology and notation which we use below. 
Let $n$ be a positive integer. 

\begin{definition}\label{definition: J lambda} 
Let $\lambda \vdash n$. We define $J_{\lambda}$ to be the set of simple positive roots associated to $\lambda$ as follows: 
\[
J_{\lambda} := \Delta \setminus\{ \alpha_{\lambda_1}, \alpha_{\lambda_1+\lambda_2}, \ldots, \alpha_{\lambda_1+\cdots + \lambda_{k-1}} \}\subseteq \Delta.
\]
\end{definition} 

 We illustrate in Example~\ref{example: standard and column} how the above definition can be visualized. Note that
any partition of $n$ corresponds to a Young diagram with $n$ boxes,  and by slight abuse of notation we denote both the partition $\lambda=(\lambda_1, \lambda_2, \cdots, \lambda_k)$ and the corresponding Young diagram as $\lambda$. We also identify the set of simple positive roots $\Delta$ with the set $[n-1]:=\{1,2,\ldots,n-1\}$ by the association $i \mapsto \alpha_i$. 

\begin{example}\label{example: standard and column} 
Let $\lambda=(5,4,4,2)\vdash 15$. Using the simplest Young tableau of this diagram which fills the boxes of $\lambda$ with the integers $\{1,2,\ldots,n\}$ in order starting from the top left and reading across rows from left to right,  starting from the top row to the bottom row, as indicated below, the set $J_\lambda=\Delta \setminus \{\alpha_5, \alpha_9, \alpha_{13}\}$ corresponds to those boxes which are not at the rightmost end of a row. In the figure below, the boxes corresponding to simple roots that are contained in $J_\lambda$ are shaded in grey. 
\[
\begin{ytableau} 
*(grey) 1 & *(grey)2 & *(grey) 3 & *(grey) 4 & 5\\
*(grey)6 & *(grey) 7 & *(grey)8  & 9\\
*(grey)10 & *(grey)11 & *(grey)12 & 13 \\
*(grey)14 & 15 \\
\end{ytableau} 
\]
\end{example} 

Recall that the dual partition $\lambda^{\vee}$ of $\lambda$ is obtained by swapping the rows and the columns of the Young diagram of $\lambda$. We will also be interested in the set $J_{\lambda^{\vee}}$ corresponding to $\lambda^{\vee}$. In fact it will be useful to introduce notation for the complement of $J_{\lambda^{\vee}}$. We let 
\begin{equation}\label{eq: def mathbb j} 
\mathbb{J}_\lambda := \Delta \setminus J_{\lambda^{\vee}}.
\end{equation}

\begin{example} 
Continuing Example~\ref{example: standard and column}, let $\lambda=(5,4,4,2)\vdash 15$. Then 
it is straightforward to see that $\lambda^{\vee} = (4,4,3,3,1)$ and $J_{\lambda^{\vee}} = \Delta \setminus \{\alpha_4, \alpha_8, \alpha_{11}, \alpha_{14}\}$ and that $\mathbb{J}_\lambda := \Delta \setminus J_{\lambda^{\vee}} = \{\alpha_4, \alpha_8, \alpha_{11}, \alpha_{14}\}$. Below, the shaded boxes in the figure on the left correspond to the positive simple roots contained in $J_{\lambda^{\vee}}$, while the shaded boxes in the figure on the right correspond to those contained in $\mathbb{J}_{\lambda} := \Delta \setminus J_{\lambda^{\vee}}$. Note that the diagram for $\lambda$ is drawn, but the labelling of the boxes corresponds to the simplest Young tableau of the dual partition $\lambda^{\vee}$. The box labelled $15$ in the diagram is contained in neither $J_{\lambda^{\vee}}$ nor $\mathbb{J}_\lambda$ since both sets are contained in $[n-1]$, not $[n]$. 
\[
\begin{ytableau} 
*(grey) 1 & *(grey)5 & *(grey) 9 & *(grey) 12 & 15\\
*(grey)2 & *(grey) 6 & *(grey)10  & *(grey)13\\
*(grey)3 & *(grey)7 & 11 & 14 \\
4 & 8 \\
\end{ytableau} 
\quad \quad 
\begin{ytableau} 
 1 & 5 & 9 & 12 & 15\\
2 & 6 & 10  & 13\\
3 & 7 & *(grey) 11 & *(grey) 14 \\
*(grey)4 & *(grey) 8 \\
\end{ytableau} 
\]
\end{example}

We will also be interested in certain subdiagrams of a Young diagram $\lambda$. First recall that for $\lambda=(\lambda_1, \cdots, \lambda_k)$ a partition with $\lambda_k > 0$, the integer $k$ is often called the number of parts of $\lambda$ (also known as the \textit{length} of $\lambda$). By definition, the number of parts of $\lambda$ is equal to $\lambda_1^{\vee}$, the first entry of the dual partition $\lambda^{\vee}$. Thus we will sometimes use the notation $\lambda_1^{\vee}$ for the number of parts. 

We will also need to refer to the number of boxes in the bottom row of $\lambda$, which is equal to $\lambda_{\lambda_1^\vee}$; however, to avoid cumbersome notation we denote this as $r(\lambda)$ and call it the \textbf{bottom length} of $\lambda$. (Thus, if $\lambda$ has $k$ parts, then $r(\lambda) = \lambda_k$.) It follows from the definitions that the maximum number of boxes in a column of $\lambda$ is exactly ${\lambda}_1^{\vee}$, and there are precisely $r(\lambda)$ many such columns in $\lambda$.

In the inductive arguments given in the later sections, we will need to remove columns from $\lambda$ as follows. 

\begin{definition}\label{definition: lambda ell} 
Let $\lambda$ be a partition of $n$. Let $\ell$ be a positive integer. Then we denote by $\lambda[\ell]$ the partition obtained by removing the leftmost $\ell$ columns from the Young diagram associated to $\lambda$. 
\end{definition} 

\begin{example} 
Let $\lambda=(6,4,2,1)$ and let $\ell=2$. Then $\lambda[2]$ is the partition $\lambda=(4,2)$ obtained by removing the leftmost $2$ columns of $\lambda$. In the figure below, the boxes that are removed are shaded, and the white boxes correspond to the smaller partition $\lambda[2]$. 
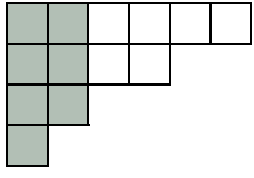
\[
\begin{ytableau}
*(grey) & *(grey) & & & & \\
*(grey) & *(grey) & & \\
*(grey) & *(grey) \\
*(grey) 
\end{ytableau}
\]
\end{example}

\begin{remark} 
Using the terminology and notation introduced above, we note that if $\lambda$ is a partition of $n$ with exactly $k$ parts and $r=r(\lambda)$ and $\ell \in \Z$ with $1 \leq \ell \leq r-1$, then the partition $\lambda[\ell]$ still has $k$ parts, while $\lambda[r]$ is a partition of $n-rk$ which has strictly fewer than $k$ parts. 
\end{remark}

\begin{definition}\label{def: step}
Let $\lambda$ be a partition. We say a consecutive sequence $\{s, s+1, \ldots, s+t\} \subseteq [\lambda_1]$ is a \textbf{step of $\lambda$} if 
\[
{\lambda}_s^{\vee} = {\lambda}_{s+1}^{\vee} = \cdots = {\lambda}_{s+t}^{\vee}
\]
and if this sequence is maximal with respect to this property, i.e., assuming the quantities are defined, both ${\lambda}_{s-1}^{\vee} \neq {\lambda}_s^{\vee}$ and ${\lambda}_{s+t+1}^{\vee} \neq {\lambda}_{s+t}^{\vee}$ (with the convention that $\lambda_0^{\vee}=0$). 
\end{definition} 
The terminology above is motivated by viewing the Young diagram of $\lambda$ as an (upside-down) staircase. 

\begin{example} 
If $\lambda=(8, 5, 3, 2)$ so that ${\lambda}^{\vee} = (4,4,3,2,2,1,1,1)$
as in the diagram below
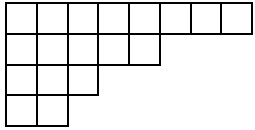
\[
\yng(8,5,3,2) 
\]
then there are four steps of $\lambda$, namely $A_1 = \{1,2\}, A_2=\{3\}, A_3 = \{4,5\}, A_4=\{6,7,8\}$. Each step gives the labels of a set of columns (starting from the left) of $\lambda$ with the same length. 
\end{example} 

It is clear that every column in $\lambda$ belongs to exactly one step of $\lambda$, giving us the following decomposition.

\begin{definition} 
The \textbf{step decomposition} of $\lambda\vdash n$ is the decomposition 
\[
[\lambda_1] = A_1 \sqcup A_2 \sqcup \cdots \sqcup A_{\mathrm{step}(\lambda)}
\]
where each $A_i$ is a step of $\lambda$ and $\mathrm{step}(\lambda)$ is a positive integer which we call the \textbf{number of steps} (or \textbf{step number}) of $\lambda$. We will always assume that the $A_i$ are listed in increasing order, i.e. $A_1 = \{1,2,\cdots, a_1\}$, $A_2 = \{a_1+1, \ldots, a_2\}$, and so on, for some sequence of integers $1 \leq a_1 < a_2 < \cdots < a_{\mathrm{step}(\lambda)} = \lambda_1$. 
\end{definition}

%%%%%%%%%%%%%%%%%%%%%
%  Linear Relations
%%%%%%%%%%%%%%%%%%%%%

\section{Linear equations satisfied by representation multiplicities}\label{sec:relations}

The main result of this section, Theorem~\ref{theorem: linear relations}, gives a set of linear equations satisfied by the multiplicity coefficients $c_\mu$ and $c_{{\mu}, i}$ of equation~\eqref{eq: decomp into Mlambda}.  In Corollary~\ref{corollary: graded matrix equation} below, we also reformulate our main result into a family of matrix equations by applying Theorem~\ref{theorem: linear relations} to the special cases when the set $J$ below is chosen to be $\mathbb{J}_{\lambda}$ for a partition $\lambda$ of $n$. 
We follow the notation introduced in Section~\ref{sec:background}.

The following sets of permutations play a key role in the analysis below. 
\begin{definition}\label{def: WJh}
Let $J\subseteq \Delta$ be any subset of the set $\Delta$ of simple positive roots and $i\in \Z$, $i\geq 0$. We define 
\[
\W_i(J,h):=\{ w\in \Symm_n \hsm\vert\hsm w^{-1}(J)\subseteq \Phi_h \textup{ and } w^{-1}(\Delta\setminus J)\subseteq I_h \textup{ and } |\inv_h(w)|=i \} \subseteq \Symm_n.
\]
We also define 
\[
\W(J,h) :=\bigsqcup_i \W_i(J,h) =  \{ w\in \Symm_n \hsm\vert\hsm w^{-1}(J)\subseteq \Phi_h \textup{ and } w^{-1}(\Delta\setminus J)\subseteq I_h \} \subseteq \Symm_n
\]
where the union is taken over all $i$ such that $\W_i(J,h)\neq \emptyset$.
\end{definition} 

It will be convenient to introduce the following notation. 
Let $w \in \Symm_n$ be a permutation. Then 
\begin{equation}\label{eq: left descents} 
\mathrm{Des}_L(w) = \{ \alpha_i\in \Delta \hsm\vert\hsm   w^{-1}(i)>w^{-1}(i+1) \} 
\end{equation} 
is the set of \textbf{left descents of $w$} and
\begin{equation}\label{eq: right descents} 
\mathrm{Des}_R(w) = \{ \alpha_i\in \Delta \hsm\vert\hsm   w(i) > w(i+1) \} 
\end{equation}
is the set of \textbf{right descents of $w$}. Both of these sets have a natural interpretation in terms of the one-line notation for $w$.  The set of left descents corresponds to the set of ordered pairs $(i,i+1)$ such that $i+1$ appears before $i$ in the one-line notation for $w$.  Similarly, the set of right descents corresponds to the pairs $(i,i+1)$  such that, in the one-line notation of $w$, the $(i+1)$-st entry is less than the $i$-th entry.

For two subsets $J$ and $K$ of $\Delta$ we define
\begin{equation}\label{eq: def D(J,K)}
\mathcal{D}(J, K) := \{ w\in \Symm_n \hsm\vert\hsm \Des_L(w)=\Delta \setminus J \textup{ and } \Des_R(w)\subseteq \Delta\setminus K \}.
\end{equation}
The goal of this section is to prove the following. 

\begin{theorem}\label{theorem: linear relations}
Let $J\subseteq \Delta$ and $i\in \Z$, $i\geq 0$.  Then
\begin{eqnarray}\label{eqn: graded linear relations}
\lvert \W_i(J, h) \rvert = \sum_{\mu\vdash n} c_{\mu, i}\; \lvert \mathcal{D}(J, J_\mu) \rvert. 
\end{eqnarray}
and
\begin{eqnarray}\label{eqn: linear relations}
\lvert \W(J, h) \rvert = \sum_{\mu\vdash n} c_{\mu}\; \lvert \mathcal{D}(J, J_\mu) \rvert. 
\end{eqnarray}
\end{theorem}

We organize this section as follows. In Section~\ref{subsec: linear relations} we prove Theorem~\ref{theorem: linear relations} modulo two elementary lemmas, and in Section~\ref{subsec: inclusion-exclusion} we record the proofs of these two lemmas. Finally, in Section~\ref{subsec: matrix equation} we re-organize a certain subset of these linear relations obtained in Theorem~\ref{theorem: linear relations}, namely, those for which $J=\mathbb{J}_\lambda$, into a set of matrix equations, one for each $i \geq 0$.

\subsection{Proof of Theorem~\ref{theorem: linear relations}}\label{subsec: linear relations}

The proof of Theorem~\ref{theorem: linear relations} relies on three results which we list below. The first is a result of Brosnan--Chow \cite{BrosnanChow2015} which relates the representation multiplicities in~\eqref{eq: decomp into Mlambda} to the Betti numbers of certain regular Hessenberg varieties.  The last two are straightforward inclusion-exclusion arguments. 

 We first state a theorem of Brosnan and Chow \cite[Theorem 127]{BrosnanChow2015}. For a given subset $J\subseteq \Delta$, let $\mathsf{X}_J\in \mathfrak{gl}(n, \C)$ be the regular element such that $\mathsf{X}_J = \mathsf{N}_J + \mathsf{S}_J$ where  
\[
\mathsf{N}_J = \sum_{\alpha_i\in J} E_{i,i+1}
\]
and $\mathsf{S}_J$ is a semisimple linear operator such that $\mathsf{N}_J$ is a regular nilpotent element in the Levi subalgebra $\mathfrak{z}_{\mathfrak{g}}(\mathsf{S}_J)$.  A Hessenberg variety associated to such a regular operator $\mathsf{X}_J$ as above is called a regular Hessenberg variety. Moreover, let $\Symm_J :=\left<s_{\alpha}: \alpha\in J\right>$ be the subgroup of the symmetric group generated by the simple reflections corresponding to the simple roots in $J$. The theorem of Brosnan and Chow identifies the dimension of the subspaces $H^{2i}(\Hess(\mathsf{S},h))^{\Symm_{J}}$ with the dimension of the cohomology of a  certain regular Hessenberg variety.

\begin{theorem}\label{thm: BrosnanChow main thm}  (Brosnan--Chow, \cite[Theorem 127]{BrosnanChow2015})
Let $n$ be a positive integer and $h: [n] \to [n]$ a Hessenberg function. Let $\mathsf{X}_{J}$ and $\Symm_J$ for $J\subseteq \Delta$ be as above, and $\mathsf{S}$ be a regular semisimple operator. Then for each non-negative integer $i$,  we have
\[
\dim (H^{2i}(\Hess(\mathsf{S}, h)))^{\Symm_{J}} = \dim H^{2i}(\Hess(\mathsf{X}_J, h)).
\]
\end{theorem}

The  next two results are straightforward inclusion-exclusion arguments which are based on 
a combinatorial formula for the Betti numbers of regular Hessenberg varieties obtained by the second author \cite{Precup2016}.

\begin{lemma}\label{lemma: inclusion-exclusion1}  Let $J\subseteq \Delta$, $h$ any Hessenberg function, and $i \in \Z, i \geq 0$. Then 
\begin{eqnarray}\label{eqn: inclusion-exclusion1}
|\W_i(J, h)| = \sum_{I\,:\,  J\subseteq I} (-1)^{|I|-|J|} \dim(H^{2i}(\Hess(\mathsf{X}_I,h))).
\end{eqnarray}
\end{lemma}

\begin{lemma}\label{lemma: inclusion-exclusion2}  Let $\mu$ be a partition of $n$ and $J\subseteq \Delta$. Then
\begin{eqnarray}\label{eqn: inclusion-exclusion2}
|\mathcal{D}(J, J_\mu)| = \sum_{I: J\subseteq I} (-1)^{|I|-|J|} \dim(M^{\mu})^{\Symm_I}.
\end{eqnarray}
\end{lemma}

We now give a proof of Theorem~\ref{theorem: linear relations}, assuming Lemma~\ref{lemma: inclusion-exclusion1} and Lemma~\ref{lemma: inclusion-exclusion2}.

\begin{proof}[Proof of Theorem~\ref{theorem: linear relations}] 
We have:
\begin{eqnarray*}
|\W_i(J, h)| &=& \sum_{I: J \subseteq I} (-1)^{|I|-|J|} \dim(H^{2i}(\Hess(\mathsf{X}_I,h))) \quad\quad \textup{by Lemma~\ref{lemma: inclusion-exclusion1}}\\
&=&  \sum_{I: J \subseteq I} (-1)^{|I|-|J|} \sum_{\mu\vdash n} c_{\mu, i} \dim(M^{\mu})^{\Symm_I} \quad \quad \textup{by Theorem~\ref{thm: BrosnanChow main thm}}\\
&=& \sum_{\mu\vdash n} c_{\mu,i}\left( \sum_{I: J \subseteq I} (-1)^{|I|-|J|} \dim(M^{\mu})^{\Symm_I} \right) \\
&=&  \sum_{\mu\vdash n} c_{\mu,i} |\mathcal{D}(J, J_{\mu})| \quad \textup{by Lemma~\ref{lemma: inclusion-exclusion2}}
\end{eqnarray*}
which proves equation~\eqref{eqn: graded linear relations}.  Equation~\eqref{eqn: linear relations} follows directly from~\eqref{eqn: graded linear relations} by summing over $i$. 
\end{proof}

\subsection{M\"obius inversion on the Boolean lattice}\label{subsec: inclusion-exclusion}

We now give proofs of the elementary lemmas used in the previous section. Both follow from an application of the well-known M\"obius inversion formula on the Boolean lattice, which is a version of the principle of inclusion-exclusion. We will need the following Betti number formula \cite[Lemma 1]{Precup2016}. 

\begin{theorem}\label{prop: betti} 
Let $J \subseteq \Delta$ and $h$ be any Hessenberg function. Then for each non-negative integer $i$,  we have
\[
{\dim}(H^{2i}(\Hess(\mathsf{X}_J, h))) = \lvert \{w \in \Symm_n\hsm\vert\hsm w^{-1}(J) \subseteq \Phi_h \textup{ and } |\inv_h(w)|=i \} \rvert. 
\]
\end{theorem}

Using the above, we first prove Lemma~\ref{lemma: inclusion-exclusion1}.

\begin{proof}[Proof of Lemma~\ref{lemma: inclusion-exclusion1}]
Let $\W_i:= \{w\in \Symm_n \hsm\vert\hsm \lvert\inv_h(w)\rvert=i \}$ and for each $I\subseteq \Delta$ define $f_I: \W_i \to \{0,1\}$ as follows: 
\begin{equation*}\label{eq: def f} 
f_I(w) = \begin{cases} 1 \quad \textup{ if } w^{-1}(I) \subseteq \Phi_h \textup{ and } w^{-1}(\Delta \setminus I) \subseteq I_h   \\ 0 \quad \textup{ else}. \end{cases} 
\end{equation*} 
For each $I\subseteq \Delta$, let us also define a function $g_I: \W_i \to \{0,1\}$ by 
\begin{equation*}\label{eq: def gI}
g_I(w) = \begin{cases} 1 \quad \textup{ if } w^{-1}(I) \subseteq \Phi_h \\ 0 \quad \textup{ else. } \end{cases}
\end{equation*} 
Then it is clear that $f_J(w) = 1$ if and only if $w \in \mathcal{W}_i(J,h)$, and thus 
\begin{equation*}\label{eq: LHS} 
\lvert \mathcal{W}_i(J,h) \rvert = \sum_{w \in \W_i} f_J(w).
\end{equation*} 
Next we examine the RHS of~\eqref{eqn: inclusion-exclusion1}. By Theorem~\ref{prop: betti} the RHS is equal to 
\[
\sum_{I: J \subseteq I} (-1)^{\lvert I \rvert - \lvert J \rvert} \lvert \{w \in \W_i \hsm\vert\hsm w^{-1}(I) \subseteq \Phi_h \} \rvert. 
\]
On the other hand, from the definition of $g_I$ it is clear that this is in turn equal to 
\begin{equation*}\label{eq: RHS}
\sum_{I: J \subseteq I} (-1)^{\lvert I \rvert - \lvert J \rvert} \sum_{w \in \W_i} g_I(w)
 = \sum_{w \in \W_i} \sum_{I: J \subseteq I} (-1)^{\lvert I \rvert - \lvert J \rvert} g_I(w).
 \end{equation*}
 Therefore, to prove the proposition it would suffice to show that 
 \begin{equation*}\label{eq: 3} 
 f_J = \sum_{I: J \subseteq I} (-1)^{\lvert I \rvert - \lvert J \rvert} g_I
 \end{equation*}
but this follows immediately from the M\"obius inversion formula on the Boolean lattice, because $ g_J = \sum_{I: J\subseteq I} f_I$ by the definitions of $g_J$ and $f_I$.  This completes the proof.
 \end{proof}

To prove Lemma~\ref{lemma: inclusion-exclusion2} we first recall the following well-known description of the numbers $\mathrm{dim}(M^{\mu})^{\Symm_{I}}$, namely: 
\begin{equation}\label{eq:1} 
\mathrm{dim}(M^\mu)^{\Symm_{I}} = \lvert \{ w \in \Symm_n \hsm\vert\hsm \mathrm{Des}_L(w) \subseteq \Delta \setminus I \textup{ and } \mathrm{Des}_R(w) \subseteq \Delta \setminus J_\mu \} \rvert.
\end{equation}

\begin{proof}[Proof of Lemma~\ref{lemma: inclusion-exclusion2}]
Consider $\mathcal{A}_{\mu} := \{w \in \Symm_n \hsm\vert\hsm \mathrm{Des}_R(w) \subseteq \Delta \setminus J_\mu \}$. On $\mathcal{A}_{\mu}$ define for each $I\subseteq \Delta$ a function $f_I: \mathcal{A}_{\mu} \to \{0,1\}$ by 
\[
f_I(w) = \begin{cases} 1 \quad \textup{ if } \mathrm{Des}_L(w) = \Delta \setminus I \\ 0 \quad \textup{ else.}  \end{cases} 
\]
On $\mathcal{A}_{\mu}$ also define for each $I \subseteq \Delta$ a function $g_I$ as follows: 
\[
g_I(w) = \begin{cases} 1 \quad \textup{ if } \mathrm{Des}_L(w) \subseteq \Delta \setminus I \\ 0 \quad \textup{ else.} \end{cases} 
\]
Then it is clear that $\lvert \mathcal{D}(J, J_\mu ) \rvert = \sum_{w \in \mathcal{A}_{\mu}} f_J(w)$ by definition of $f$. 

We now examine the RHS of~\eqref{eqn: inclusion-exclusion2}. We have 
\begin{equation*}
\begin{split} 
\mathrm{RHS} & = \sum_{I: J \subseteq I} (-1)^{\lvert I \rvert - \lvert J \rvert} \dim(M^{\mu})^{\Symm_I} \\
& = \sum_{I: J \subseteq I} (-1)^{\lvert I \rvert - \lvert J \rvert} \; \big\lvert \{ w \in \Symm_n \hsm\vert\hsm \mathrm{Des}_L(w) \subseteq \Delta \setminus I \textup{ and } \mathrm{Des}_R(w) \subseteq \Delta \setminus J_\mu \} \big\rvert  \quad \textup{ by~\eqref{eq:1} }   \\
& = \sum_{I: J \subseteq I} (-1)^{\lvert I \rvert - \lvert J \rvert} \sum_{w \in \mathcal{A}_{\mu}} g_I(w) \\
& = \sum_{w \in \mathcal{A}_{\mu}} \sum_{I: J \subseteq I} (-1)^{\lvert I \rvert - \lvert J \rvert} g_I(w).
\end{split} 
\end{equation*}
Thus it suffices to show that 
\begin{equation*}\label{eq: sum of gI}
f_J(w) = \sum_{I: J \subseteq I} (-1)^{\lvert I \rvert - \lvert J \rvert} g_I
\end{equation*}
but, as in the proof of the previous lemma, this follows immediately from the 
M\"obius inversion formula on the Boolean lattice, since $ g_J = \sum_{I: J\subseteq I} f_I$ from the definitions of $g_J$ and $f_I$. This completes the proof.    
\end{proof}

\subsection{A new matrix equation}\label{subsec: matrix equation}

We now introduce the matrix equation that is the subject of the next section.  We will be particularly interested in $\W_i(J,h)$ in the case that $J = \mathbb{J}_\lambda$ and we introduce notation for the cardinality of the sets in~\eqref{eq: def D(J,K)} for the case $J=\mathbb{J}_\lambda$ and $K=J_\mu$ for two partitions $\lambda, \mu \vdash n$. Let
\begin{equation}\label{eq: def A lambda mu}
A(\lambda, \mu) := \lvert \mathcal{D}(\mathbb{J}_\lambda, J_\mu) \rvert.
\end{equation}
Using the above notation, Theorem~\ref{theorem: linear relations} can be rewritten as follows. 
Let $\Par(n)$ denote the set of partitions of $n$.

\begin{corollary}\label{corollary: graded matrix equation} 
 Let $A = (A(\lambda, \mu))_{\lambda,\mu \in \Par(n)}$ be the matrix whose coefficients are the integers~\eqref{eq: def A lambda mu} and $i\in \Z$, $i\geq 0$. Let $X_i$ be the (column) vector whose entries are the $c_{\mu,i} \in \Z$ specified in~\eqref{eq: decomp into Mlambda}. Let $W_i$ be the (column) vector whose entries are the integers  $\lvert \W_i(\mathbb{J}_\lambda,h) \rvert$. Then $AX_i=W_i$. 
\end{corollary}

Note that the indexing set for the matrix entries of $A$ is the set $\Par(n)$ of all partitions of $n$. In the next section we will show that the matrix $A$ has computationally convenient properties with respect to an appropriate choice of total order on $\Par(n)$.

%%%%%%%%%%%%%%%%%%%%%
%  Upper Triangular
%%%%%%%%%%%%%%%%%%%%%

\section{Upper-triangularity of $A$ and an inductive formula for the matrix entries}\label{sec:upper triangular}

In the previous section, we saw that the multiplicity coefficients $c_{\mu, i}$ in~\eqref{eq: decomp into Mlambda} obey a set of linear equations which can be interpreted as a matrix equation $AX_i=W_i$. Moreover, since there exists such a linear equation for each choice of a partition $\lambda \vdash n$, and since the indexing set of the coefficients $c_{\mu,i}$ is also the set of partitions of $n$, the matrix $A=(A(\lambda,\mu))$ is in fact a square matrix.  

The main results of this section are Proposition~\ref{proposition: M-formulas Martha version} and Theorem~\ref{theorem: upper-triangular}. Proposition~\ref{proposition: M-formulas Martha version} states that certain matrix entries of $A$ have an inductive description or are equal to $0$. Theorem~\ref{theorem: upper-triangular} states that -- with respect to an appropriately defined total order on the set of partitions of $n$ -- the matrix $A$ is upper-triangular with $1$'s along the diagonal.

We begin by stating the first main result. Recall that $\lambda[\ell]$ denotes the partition obtained by deleting $\ell$ columns from $\lambda$ as in Definition~\ref{definition: lambda ell}.

\begin{proposition}\label{proposition: M-formulas Martha version} Let $\lambda\vdash n$ be a 
partition with exactly $k$ parts and let $r=\lambda_k$ be the bottom length of $\lambda$.  Let $\mu\vdash n$ be a partition of $n$ with at most $k$ parts.  Then
\begin{enumerate} 
\item if $\mu_k < \lambda_k$ (so in particular if $\mu_k=0$, i.e., $\mu$ has strictly fewer than $k$ parts), then 
\[
 \mathcal{D}(\mathbb{J}_\lambda, J_{\mu})  = \emptyset \; \textup{ and therefore } \; 
 A(\lambda, \mu) =0
 \]
 and 
 \item if $\mu_k \geq \lambda_k$, then for any $\ell \in \Z$ with $0 \leq \ell \leq r,$ there exists a natural bijection between the sets 
 \[
 \mathcal{D}(\mathbb{J}_\lambda, J_{\mu}) \textup{ and } \mathcal{D}(\mathbb{J}_{\lambda[\ell]}, J_{\mu[\ell]})
 \]
 and in particular we have 
 \[
 A(\lambda, \mu) = A(\lambda[\ell], \mu[\ell]).
 \]
\end{enumerate} 
\end{proposition}

We will prove Proposition~\ref{proposition: M-formulas Martha version} in due course, but we first state the second main result of this section, which is an upper-triangularity property of the matrix $A$. First, we define an appropriate total order on the set of partitions of $n$. 

\begin{definition}\label{definition: lex on dual} Let $n$ be a positive integer and let $\Par(n)$ denote the set of partitions of $n$.  We define a total ordering $\preceq$ on $\Par(n)$ as follows:
\begin{equation}\label{eq: lex on dual}
\mu \preceq \lambda \Leftrightarrow \mu^{\vee} \leq_{\textup{lex}} \lambda^{\vee}.
\end{equation}
\end{definition}

\begin{example} Let $n=6$ and consider $\lambda=(3,3)$ and $\mu=(4,1,1)$. Note that $\lambda$ and $\mu$ are incomparable in the dominance order, but $\lambda^{\vee}=(2,2,2)$ and $\mu^{\vee}=(3,1,1,1)$ so $\lambda^{\vee}<_{\textup{lex}} \mu^{\vee}$ and therefore, according to our definition~\eqref{eq: lex on dual}, we have $\lambda \prec \mu$.
\end{example}

\begin{remark} 
It is straightforward to see that lexicographical order of $\Par(n)$, which is a total order, respects the dominance (partial) ordering on $\Par(n)$, in the sense that $\mu \unlhd \lambda$ implies $\mu \leq_{\textup{lex}} \lambda$.  It is also well known that $\mu \unlhd \lambda$ if and only if their dual partitions satisfy the reverse relation, i.e. $\lambda^{\vee} \unlhd \mu^{\vee}$.  It follows that the total order $\preceq$ of Definition~\ref{definition: lex on dual} on $\Par(n)$ respects the reversed dominance order. 
\end{remark} 

We now state our upper-triangularity theorem. 

\begin{theorem}\label{theorem: upper-triangular}
The matrix $(A(\lambda,\mu))_{\lambda, \mu \in \Par(n)}$, written with respect to the total order~\eqref{eq: lex on dual} on the indexing set $\Par(n)$, is upper-triangular with $1$'s along the diagonal. Equivalently, for $\lambda, \mu\in \Par(n)$, we have the following: 
\begin{enumerate}
\item If $\mu \prec \lambda$ with respect to the total order~\eqref{eq: lex on dual} then $\mathcal{D}(\mathbb{J}_\lambda, J_{\mu}) = \emptyset$, so in particular, $A(\lambda, \mu) = 0$. 
\item The set $\mathcal{D}(\mathbb{J}_\lambda, J_\lambda)$ contains a unique element, so in particular, $A(\lambda, \lambda)=1$. 
\end{enumerate}
\end{theorem}

\begin{example} When $n=2$ we get the matrix: 
\[
(A(\lambda,\mu))_{\lambda,\mu\in \Par(2)} = \begin{bmatrix} A((2),(2)) & A((2),(1,1)) \\ A((1,1),(2)) & A((1,1),(1,1)) \end{bmatrix} = \begin{bmatrix} 1 & 1\\ 0 & 1 \end{bmatrix}
\]
and similarly for $n=3$ we have $\Par(3) = \{ (3) \prec (2,1) \prec (1,1,1) \}$ and it can be checked directly that we get the matrix 
\[
(A(\lambda, \mu))_{\lambda,\mu \in \Par(3)}  = \begin{bmatrix} A((3),(3)) & A((3),(2,1)) & A((3),(1,1,1)) \\ A((2,1),(3))  & A((2,1),(2,1)) & A((2,1),(1,1,1))  \\ A((1,1,1),(3)) & A((1,1,1),(2,1))  & A((1,1,1),(1,1,1)) \end{bmatrix}  = \begin{bmatrix} 1 & 1 & 1 \\ 0 & 1 & 2 \\ 0 & 0 & 1 \end{bmatrix}.
\]
\end{example}

The remainder of this section is devoted to the proofs of Proposition~\ref{proposition: M-formulas Martha version} and Theorem~\ref{theorem: upper-triangular}. We need several preliminaries.  \textit{For what follows, we frequently identify $\Delta$ with the set $[n-1] = \{1, 2, \ldots, n-1\}$ using the bijection $i \leftrightarrow \alpha_i$.}  

Let $J = \{i_1 < i_2 < \cdots < i_\ell \} \subseteq \Delta \cong [n-1]$ be a subset of the positive simple roots. The \textbf{staircase decomposition (of $[n]$) corresponding to $J$} is the decomposition 
\[
[n] = \{i_0=1 ,2,\ldots, i_1\} \sqcup \{i_1+1, i_1 + 2, \cdots, i_2\} \sqcup \cdots \sqcup \{i_{\ell-1}+1, \cdots, i_{\ell}\} \sqcup \{i_{\ell}+1, \cdots, n=i_{\ell+1}\}.
\]
where by convention we set $i_0:=1$ and $i_{\ell+1}:=n$. Each subset appearing in the above decomposition is called a \textbf{staircase}. The motivation for the ``staircase'' terminology comes from studying the set of right descents of a permutation $w\in \Symm_n$.  It follows directly from the definition of $\Des_R(w)$ in~\eqref{eq: right descents} that if $\Des_R(w)\subseteq J=\{i_1< i_2< \cdots< i_{\ell} \} \subseteq \Delta\simeq [n-1]$ then for all $0\leq s\leq \ell$ we have
\[
w(i_s+1) < w(i_s+2) < \cdots < w(i_{s+1})
\]
on each staircase $\{i_s+1, i_s+2,\ldots, i_{s+1}\}$ of $J$. 

Given a subset $J$ of cardinality $\ell$, the \textbf{number of staircases} in its associated staircase decomposition, which we denote $\mathbb{F}(J)$, is $\ell+1$, i.e. 
\[
\mathbb{F}(J) := \lvert J \rvert +1 = \ell +1.
\]

We also find it convenient to introduce analogous terminology for the permutations themselves. 
Let $w \in \Symm_n$ and $\{i_s+1, i_s+2, \cdots, i_{s+1}\} \subseteq [n-1]$ for $i_{s+1}>i_s$ be a sequence of consecutive integers, possibly of length $1$ (when $i_{s+1} = i_s+1$). We say $w$ \textbf{is a staircase on the interval $\{i_s+1, i_s+2, \cdots, i_{s+1}\}$}
if $w(i_s+1) < w(i_s+2) < \cdots < w(i_{s+1})$. We also say that $\{i_s+1, i_s+2, \cdots, i_{s+1}\}$ \textbf{is a staircase of $w$}. A staircase $\{i_s+1, i_s+2, \cdots, i_{s+1}\}$ of $w$ is \textbf{maximal} if neither $\{i_s, i_s+1, i_s+2, \cdots, i_{s+1}\}$ nor $\{i_s+1, i_s+2, \cdots, i_{s+1}, i_{s+1}+1\}$ is a staircase of $w$. 
The following is immediate from the definition of the right descent set given in~\eqref{eq: right descents} and we omit the proof. 

\begin{lemma}\label{lemma: staircase}
Let $w \in \Symm_n$. 
Suppose $J = \{i_1 < i_2 < \cdots < i_{\ell} \}$. Let $i_0 := 1$ and $i_{\ell+1} :=n$. If $\mathrm{Des}_R(w) \subseteq J$, then $w$ is a staircase on each interval $\{i_s+1, i_s+2, \cdots, i_{s+1}\}$ for $0 \leq s \leq \ell$, and there are at most $\ell+1$ maximal staircases in the staircase decomposition of $w$.
In particular, suppose $\mu = (\mu_1, \ldots, \mu_k)$ is a partition of $n$ with $k$ parts, and $\mathrm{Des}_R(w) \subseteq  \Delta \setminus J_{\mu} = \{\mu_1, \mu_1+\mu_2, \cdots, \mu_1 + \cdots + \mu_{k-1}\}$. Then there are at most $\mathbb{F}(\Delta \setminus J_{\mu}) = k$ maximal staircases of $w$. 
\end{lemma}

\begin{example}\label{example: staircase}
Let $w = [1,4,7,8, 2,5,6,3] \in \Symm_8$. Then $\mathrm{Des}_R(w) = \{4, 7\}$ since it is between the $4$th and $5$th entries, as well as the $7$th and $8$th entries, that there is a decrease in the one-line notation of $w$.  The maximal staircases of $w$ are $\{1,2,3,4\}$, $\{5,6,7\}$ and $\{8\}$. Note that $\Des_R(w)\subseteq \Delta\setminus J_{\mu}$ where $\mu = (4,3,1)$. In this case, $\mathbb{F}(\Delta \setminus J_{\mu}) = \mathbb{F}(\{4,7\}) = 3$ is the number of maximal staircases of $w$, in agreement with the lemma above. 
\end{example}

We now turn our attention to left descents. As already noted, for a permutation $w \in \Symm_n$, the left descent set $\mathrm{Des}_L(w)$ specifies which pairs of the form $(i,i+1)$, for $1 \leq i \leq n-1$, have the property that $i+1$ appears to the left of the $i$ in the one-line notation of $w$. 
Let $w \in \Symm_n$ and $i \in [n]$. For a given staircase of $w$, we say $i$ \textbf{occurs in that staircase} if $i$ appears in the segment of the one-line notation of $w$ corresponding to that staircase.

\begin{example}
Continuing with Example~\ref{example: staircase}, let $w = [1,4,7,8, 2,5,6,3] \in \Symm_8$. Then $\{1,2,3,4\}$ is a staircase, and we say that $7$ appears in that staircase since $7$ occurs as one of the entries in positions $1$, $2$, $3$, or $4$ in the one-line notation of $w$.
\end{example} 

Note that any $j \in [n]$ occurs in exactly one maximal staircase of $w$ for any $w \in S_n$.  From the definition of staircases and left descents, the following is straightforward. 

\begin{lemma}\label{lemma: bound} 
Let $w \in \Symm_n$. Suppose that $\{j, j+1, \cdots, j+\ell-1\} \subseteq \mathrm{Des}_L(w)$ is a sequence of $\ell$ consecutive integers contained in $\mathrm{Des}_L(w)$. Then the $\ell+1$ many integers $j+\ell  > j+\ell-1 > \cdots > j+1 > j$ must appear in distinct maximal staircases of $w$, each strictly to the right of the previous one. In particular, the number of maximal staircases of $w$ must be greater than or equal to $\ell+1$. 
\end{lemma} 

\begin{proof} 
Within each staircase, the entries in the one-line notation of $w$ must be increasing, so any pair of consecutive integers which must appear in inverted order cannot appear in the same staircase. Moreover, if they must be inverted, then the smaller integer must appear to the right of the greater integer i.e., must appear in a staircase strictly to the right of the greater integer. 
\end{proof}

The  next statement follows from Lemmas~\ref{lemma: staircase} and~\ref{lemma: bound}.

\begin{corollary}\label{corollary: empty} 
Suppose $K \subseteq [n-1] \cong \Delta$ is a subset of $[n-1] \cong \Delta$ containing a consecutive sequence of simple  roots of length $\ell$. Let $\mu=(\mu_1,\ldots,\mu_k)$ be a partition of $n$ with $k$ parts. Then the set 
\begin{equation}\label{eq: D K Jmu}
\mathcal{D}(\Delta\setminus K, J_{\mu}) = \{ w \in \Symm_n \mid \mathrm{Des}_L(w) = K \textup{ and } \mathrm{Des}_R(w) \subseteq \Delta \setminus J_{\mu} \}
\end{equation}
is empty if $\ell+1 > k $. 
\end{corollary} 

\begin{proof} 
Suppose $w \in \Symm_n$ and that $\mathrm{Des}_L(w)=K$. Since $K$ contains a sequence of $\ell$ many consecutive simple positive roots, from Lemma~\ref{lemma: bound} it follows that the number of maximal staircases of $w$ is at least $\ell+1$. On the other hand, if $\mathrm{Des}_R(w) \subseteq \Delta \setminus J_\mu$ then by Lemma~\ref{lemma: staircase}, we have $\mathbb{F}(\Delta \setminus J_\mu) = k$, and $w$ has at most $k$ maximal staircases. Since $\ell+1 > k$, this cannot occur. Hence~\eqref{eq: D K Jmu} is empty as desired. 
\end{proof} 

In fact, we can say more. The following statement is straightforward  and we omit the proof.

\begin{lemma}\label{lemma:distribute in staircases}
Let $\mu$ be a partition of $n$ with $k$ many parts. Let $w \in \Symm_n$ and suppose $\mathrm{Des}_L(w)$ contains a sequence  $\{a, a+1, \ldots, a+k-2\} \subseteq [n-1]\simeq \Delta$ of maximal cardinality $k-1$ and $\mathrm{Des}_R(w) \subseteq \Delta \setminus J_{\mu}$. Then: 
\begin{enumerate} 
\item $\mathrm{Des}_R(w) = \Delta \setminus J_{\mu}$, so in particular the one-line notation of $w$ contains precisely $k$ maximal staircases, and 
\item for each $i$ such that $0 \leq i \leq k-1$, the element $a+i$ in the sequence $\{a, a+1, \ldots, a+k-1\}$ must appear in the $(i+1)$st staircase of the one-line notation of $w$ (counting from the left). 
\end{enumerate} 
In particular, if the hypotheses are satisfied, then the staircases in which each $a+i$ must occur is fixed, and exactly one element in the sequence $\{a, a+1, \ldots, a+k-1\}$ occurs in each of the $k$ maximal staircases. 
\end{lemma}

In the course of the argument below it will be useful to have the following terminology. Suppose $w \in \Symm_n$ and suppose $m \in \Z$, $1 \leq m < n$. There is a map (which is not a group homomorphism) 
\[
d_{n,m}: \Symm_n \to \Symm_{n-m} 
\]
obtained by deleting the entries $\{1,2,\ldots,m\}=[m]$ in the one-line notation of $w$, and interpreting what remains as a permutation of $n-m$, under the identification $\{m+1,m+2,\ldots, n\} \cong \{1,2,\ldots,m\}$ given by $j \mapsto j-m$. We will refer to this procedure of applying $d_{n,m}$ as \textbf{ignoring the $[m]$ entries} (of the one-line notation of $w$). 

\begin{example} 
Let $m=2$ and $n=5$. Let $w = [4, 3, 2 , 5 , 1]$. Then $d_{5,2}(w) = [2 , 1 , 3]$ because we first ignore the entries $1$ and $2$ in $w = [4 , 3 , \mathbf{2} , 5 , \mathbf{1}]$ to obtain $[4 , 3 , 5]$ and then use the identification $j \mapsto j-2$ to obtain $[2 , 1 , 3]$. 
\end{example} 

We are now ready to prove Proposition~\ref{proposition: M-formulas Martha version}. 

\begin{proof}[Proof of Proposition~\ref{proposition: M-formulas Martha version}] 
We begin with the case $\mu_k < \lambda_k$, which itself can be separated into two subcases, namely, $\mu_k=0$ and $0 < \mu_k < \lambda_k$. First suppose $\mu_k=0$, i.e., $\mu$ has strictly fewer than $k$ parts. 
From the definition of the set $\mathbb{J}_\lambda$, it follows that there are $r=\lambda_k$ many distinct sequences in $\Delta \setminus \mathbb{J}_\lambda = J_{\lambda^{\vee}}$, of the form 
\[
\{1, 2, \ldots k-1\}, \{k+1, k+2, \cdots, 2k-1\}, \cdots, \{(r-1)k+1, \ldots, kr-1\}.
\]
This means in particular that the set $\mathbb{J}_\lambda$ contains at least one consecutive sequence of positive simple roots of length $r-1$. Applying Corollary~\ref{corollary: empty}, we immediately obtain that $\mathcal{D}(\mathbb{J}_\lambda, J_{\mu}) = \emptyset$ if $\mu$ has strictly fewer than $k$ parts. This proves the proposition in the case $\mu_k=0$.

Next we consider the case when $\mu$ has $k$ parts but $\mu_k < \lambda_k$. Suppose that $w \in \mathcal{D}(\mathbb{J}_\lambda, J_{\mu})$, so $\mathrm{Des}_L(w) = \mathbb{J}_\lambda$ and $\mathrm{Des}_R(w) \subseteq \Delta \setminus J_{\mu}$. 
Then $w$ satisfies the hypotheses of Lemma~\ref{lemma:distribute in staircases} and it follows that the given conditions completely determine the staircases in which the integers $\{1,2,\ldots,kr\}$ must occur in the one-line notation of $w$. In fact, since these are the smallest $kr$ integers in $[n]$ and since each staircase must have increasing entries, the hypotheses determine the precise \emph{location} (not just the staircase) in which these entries must occur. In particular, the $r$ many integers $\{1, k+1, 2k+1, \ldots, (r-1)k + 1\}$ must appear in the rightmost staircase of $w$, which contains $\mu_k$ many entries. This implies that $\mu_k \geq r=\lambda_k$, or in other words, if $\mu_k < \lambda_k$ then $\mathcal{D}(\mathbb{J}_\lambda , J_{\mu}) = \emptyset$. This concludes the proof of part (1) of the proposition. 

Now suppose that $\mu_k \geq r = \lambda_k$. By similar reasoning as in the previous paragraph, it follows that if a permutation $w \in \Symm_n$ satisfies $\mathrm{Des}_L(w) = \Delta \setminus \mathbb{J}_\lambda=J_{\lambda^{\vee}}$ and $\mathrm{Des}_R(w) \subseteq \Delta \setminus J_\mu$, then $w$ is determined by the location (in the one-line notation) of the integers $\{kr+1,kr+2,\ldots,n\} \cong [n-kr]$, i.e., the image of $w$ under the map $d_{n,kr}$ described above. It is straightforward to see that $w$ is also determined by its image under the map $d_{n, k \ell}$ for any $1 \leq \ell \leq r$.  In what follows, for concreteness we make the argument in detail for the special case $\ell=r$.  Consider the image in $\Symm_{n-kr}$ of the set 
\begin{equation}\label{eq: descent L and R}
\mathcal{D}(\mathbb{J}_\lambda, J_{\mu}) = \{w \in \Symm_n \hsm\vert\hsm \mathrm{Des}_L(w) =\Delta\setminus \mathbb{J}_\lambda \textup{ and } \mathrm{Des}_R(w) \subseteq \Delta\setminus J_\mu\} 
\end{equation}
under the map $d_{n,kr}$ which ignores the $[kr]$ entries. By the above argument, $d_{n, kr}$ is injective on~\eqref{eq: descent L and R}. To prove the desired claim, it suffices to show that the image of~\eqref{eq: descent L and R} under $d_{n,kr}$ is precisely 
\begin{equation}\label{eq: image} 
\mathcal{D}(\mathbb{J}_{\lambda[r]}, J_{\mu[r]}) =  \{w' \in \Symm_{n-kr} \hsm\vert\hsm \mathrm{Des}_L(w') = \Delta_{n-kr} \setminus \mathbb{J}_{\lambda[r]}  \textup{ and } \mathrm{Des}_R(w') \subseteq \Delta_{n-kr} \setminus J_{\mu[r]} \}
\end{equation}
where we temporarily denote by $\Delta_{n-kr}$ the set of positive simple roots corresponding to $\mathfrak{gl}_{n-kr}(\C)$. To see this, we first show that any $w' = d_{n, kr}(w)$ for $w$ in~\eqref{eq: descent L and R} must lie in~\eqref{eq: image}. Since $\mathrm{Des}_L(w) = \Delta \setminus \mathbb{J}_\lambda={J}_{\lambda^{\vee}}$, we already know that the left descents $(j, j+1)$ occurring in $w$ with $j > kr$ are precisely the ones of the form
\[
\{kr+1,kr+2,\ldots,n\} \setminus \{ kr+\lambda_{r+1}^{\vee}, kr+\lambda_{r+1}^{\vee} + \lambda_{r+2}^{\vee}, \ldots, kr+\lambda_{r+1}^{\vee}+\cdots+\lambda_{t-1}^{\vee} \} 
\]
where $\lambda^{\vee}=(\lambda_1^{\vee}, \lambda_2^{\vee}, \ldots, \lambda_t^{\vee})$ has $t$ parts and $\lambda_1^{\vee}=\cdots= \lambda_r^{\vee}=k$ by assumption.  Notice that $\lambda[r]^{\vee} = (\lambda_{r+1}^{\vee},\lambda_{r+2}^{\vee}, \ldots, \lambda_t^{\vee})$. Under the identification of $\{kr+1,kr+2,\ldots,n\}$ with $[n-kr]$ given by $j \mapsto j-kr$, this means that $w'$ has left descent set $\Delta_{n-kr} \setminus \mathbb{J}_{\lambda[r]}$. 

Next we need to show that $\mathrm{Des}_R(w') \subseteq \Delta_{n-kr} \setminus J_{\mu[r]}$. It follows from the above that the entries $\{1,2,\ldots, kr\}$ distribute themselves in the $k$ staircases of the one-line notation of $w$ in such a way that each staircase contains precisely $r$ many of the entries within $\{1,2,\ldots, kr\}$. Therefore, when ignoring the $[kr]$ entries in $w$ to obtain $w'$, the locations where the right descents can possibly occur are precisely at 
\[
 \{\mu_1-r, \mu_1+\mu_2-2r, \ldots, \mu_1+ \cdots+\mu_{k-1}-(k-1)r  \} 
\] 
which is exactly the set $\Delta_{n-kr} \setminus J_{\mu[r]}$ for the partition $\mu[r] = (\mu_1-r, \mu_2-r, \ldots, \mu_k-r)$. In particular we conclude \(\mathrm{Des}_R(w') \subseteq \Delta_{n-kr} \setminus J_{\mu[r]}\) as desired. 

Thus $d_{n,kr}$ sends the set~\eqref{eq: descent L and R} into the set~\eqref{eq: image}. In fact, the argument given above is reversible, i.e., any $w' \in \Symm_{n-kr}$ lying in~\eqref{eq: image} can be extended to an element in $\Symm_n$ by reversing the correspondence to $j \mapsto j+kr$ and adding the entries $\{1,2,\ldots, kr\}$ in exactly the locations specified by the hypotheses in~\eqref{eq: descent L and R}, and it is clear that this extension then lies in~\eqref{eq: descent L and R}. This proves the claim in the special case $\ell=r$.  For any $1 \leq \ell < r$, by arguments similar to those above it follows that the entries of $d_{n, \ell k}(w)$ corresponding to the integers $\{1,2,\ldots, (r-\ell)k\}$ are already determined, and so an argument essentially identical to the one above proves the desired claim. This concludes the proof of the proposition. 
\end{proof}

From Proposition~\ref{proposition: M-formulas Martha version} we readily obtain the following.

\begin{corollary}\label{corollary: truncate A} 
Let $\lambda, \mu$ be partitions of $n$ and suppose that there exists $\ell \in \Z, \ell \geq 1$, such that the dual partitions $\lambda^{\vee}$ and $\mu^{\vee}$ agree up to the $\ell$-th entry, i.e. $\lambda^{\vee}_s = \mu^{\vee}_s$ for all $1 \leq s \leq \ell$. Then 
\[
A(\lambda, \mu) = A(\lambda[\ell], \mu[\ell]). 
\]
\end{corollary}

\begin{proof} 
The argument is a simple induction on the number of steps (in the sense of Definition~\ref{def: step}) in the partitions $\lambda$ and $\mu$ on which they agree. More precisely, suppose 
\[
[\lambda_1] = A_1 \sqcup A_2 \sqcup \cdots \sqcup A_{\mathrm{step}(\lambda)}
\]
is the step decomposition of $\lambda$ and define $u$ to be the index of the step in which $\ell$ occurs, i.e., 
suppose $\ell \in A_u$. 

We take cases. Suppose $u=1$. Let $k$ denote the number of parts of $\lambda$. Then the Young diagrams of $\lambda$ and $\mu$ both contain as their leftmost $\ell$ columns a rectangular $k \times \ell$ box. Proposition~\ref{proposition: M-formulas Martha version} then implies $A(\lambda, \mu) = A(\lambda[\ell], \mu[\ell])$ as desired. This proves the base case. 

Now suppose $u > 1$. Also suppose by induction that the claim is proved for $u-1$. Since $u>1$ we know $\lambda$ and $\mu$ both contain a rectangular $k \times r$ box where $k$ is the number of parts of both $\lambda$ and $\mu$ and $r=\lambda_k$ is the bottom length of both $\lambda$ and $\mu$. Another application of Proposition~\ref{proposition: M-formulas Martha version} implies that $A(\lambda, \mu) = A(\lambda[r], \mu[r])$. By assumption, the dual partitions of $\lambda[r]$ and $\mu[r]$ agree up to entry $\ell -r$, and in the step decomposition of $\lambda[r]$, the number $\ell-r$ occurs in step $A_{u-1}$ since we have deleted a full step from $\lambda$ to obtain $\lambda[r]$. Hence by induction
we know $A((\lambda[r])[\ell-r], (\mu[r])[\ell-r]) = A(\lambda[r], \mu[r])$, but from Definition~\ref{definition: lambda ell} it is clear that $\nu[s][t] = \nu[s+t]$ for any partition $\nu$ and $s, t$ for which the statement makes sense, so the result follows. 
\end{proof}

We are finally in a position to prove the upper-triangularity property.

\begin{proof}[Proof of Theorem~\ref{theorem: upper-triangular}]
 Since $\lambda^{\vee} >_{\textup{lex}} \mu^{\vee}$, there exists some $\ell \in \Z_{\geq 1}$ such that $\lambda^{\vee}_s = \lambda^{\vee}_s$ for all $0 \leq s \leq \ell$ and $\lambda^{\vee}_{\ell+1} > \mu^{\vee}_{\ell+1}$. (If no such $\ell$ exists, then $\lambda_1>\mu_1$ and we may apply Proposition~\ref{proposition: M-formulas Martha version} directly.) By Corollary~\ref{corollary: truncate A}, we know $A(\lambda, \mu) = A(\lambda[\ell], \mu[\ell])$. By construction, $\lambda[\ell]$ and $\mu[\ell]$ have the property that $(\lambda[\ell]^{\vee})_1 > (\mu[\ell]^{\vee})_1$. Hence by Proposition~\ref{proposition: M-formulas Martha version}, we have $A(\lambda[\ell], \mu[\ell]) = 0$, as desired. 
 
 We also need to show that for any $\lambda$, we have $A(\lambda, \lambda)=1$. Indeed, applying the Corollary~\ref{corollary: truncate A} to $\ell=\lambda_1 -1$ we obtain that $A(\lambda, \lambda) = A(\lambda[\lambda_1-1], \lambda[\lambda_1-1])$. By construction, $\lambda[\lambda_1-1]$ is a partition with only one column. Therefore we are now reduced to showing that if a partition $\nu$ is of the form $\nu=(1,1,\cdots, 1)$ then $A(\nu, \nu)=1$. Let $\nu$ be such a be a partition of $m$ for some positive integer $m \leq n$. By definition, $\mathbb{J}_\nu = \emptyset = J_\nu$ so $\Delta \setminus \mathbb{J}_\nu = \{\alpha_1, \ldots, \alpha_{m-1}\} = \Delta \setminus J_\nu$. This means $\mathcal{D}(\mathbb{J}_\nu, J_{\nu})$ consists of permutations $w$ in $\Symm_m$ with the property that every pair $(i,i+1)$ for all $1 \leq i \leq m-1$ appears inverted in the one-line notation of $w$, and that for all $i$ such that $1 \leq i \leq m-1$, we have $w(i) > w(i+1)$. The only such permutation is the longest element $[m,m-1, \ldots, 2, 1]\in \Symm_m$, so $\mathcal{D}(\mathbb{J}_\nu, J_{\nu})$ is a singleton set and $A(\nu, \nu)=1$ as desired. This concludes the proof. 
\end{proof}

%%%%%%%%%%%%%%%%%%%%%
%  An Inductive Formula (Sink Sequences)
%%%%%%%%%%%%%%%%%%%%%

\section{An inductive formula for the $W$-vector}\label{sec: inductive formula}

We saw in Section~\ref{subsec: matrix equation} that the coefficients $c_{\mu,i}$ of~\eqref{eq: decomp into Mlambda}, when written as a column vector $X_i=(c_{\mu,i})$, satisfy a matrix equation $AX_i = W_i$.  In order to solve this matrix equation, we need to analyze the ``constant vector'' $W_i$ for each $i$.  This is the purpose of this section.

Recall that the vector $W_i$ is defined to have entries $\lvert \W_i(\mathbb{J}_{\lambda},h) \rvert$, where the sets $\W_i(\mathbb{J}_{\lambda},h)$ are introduced in Definition~\ref{def: WJh}, and $\lambda$ varies over the partitions of $n$. 
The main result  (Theorem~\ref{thm: inductive formula}) of this section is an inductive description of the set $\W_i(\mathbb{J}_{\lambda} , h)$ in the case that $\lambda$ has $k=ht(I_h)+1$ parts.   However, it is worthwhile to note that the assumption that $\lambda$ has exactly $k=ht(I_h)+1$ parts will \emph{not} be required for many of the other results in this section.

%%%%%%%%%%%%%%%

\subsection{Sink sets and subsets of height $k$}

 In order to obtain our inductive formula, we exploit the structural relationship between the ideal $I_h$ and graph $\Gamma_h$ alluded to in Section~\ref{sec:background}.  Recall the following notation from \cite{HaradaPrecup2017}.
\begin{itemize}
\item We let $\A(\Gamma_h)$ denote the set of all acyclic orientations of $\Gamma_h$ and $\A_k(\Gamma_h)$  denote the set of all acyclic orientations with exactly $k$ sinks.
\item Given $\omega\in \A(\Gamma_h)$ we denote the subset of vertices that occur as sinks of $\omega$ by $\sk(\omega)$.  Note that each independent set of vertices in $\Gamma_h$ occurs as the sink set of some acyclic orientation and $\sk(\omega)$ is independent for each $\omega\in \A(\Gamma_h)$.
\item Let $\SK_k(\Gamma_h)$ be the set of all  possible sink sets (or, independent sets)  of $\Gamma$ of cardinality $k$.  
\item The maximum sink set size $m(\Gamma_h)$ is the maximum of the cardinalities of the sink sets $\sk(\omega)$ associated to all possible acyclic orientations of $\Gamma_h$, i.e., 
\[
m(\Gamma_h) := \mathrm{max} \{ \lvert \sk(\omega) \rvert \, \vert \, \omega \in \A(\Gamma_h) \}.
\]
\end{itemize}
The \textbf{sink set decomposition} is 
\begin{eqnarray}\label{eqn: sink set decomposition}
\A_k(\Gamma_h) = \bigsqcup_{T\in \SK_k(\Gamma_h)} \{ \omega\in \A_k(\Gamma_h) \, \vert\, \sk(\omega)=T \}.
\end{eqnarray}
With this terminology in place, our goal is to extend the sink set decomposition  of $\A_k(\Gamma_h)$ to a sink set decomposition of the set $\W(\mathbb{J}_\lambda, h)$.

If $T\in \SK(\Gamma_h)$ let $\Gamma_h[T]:= \Gamma_h - T$ be the graph obtained from $\Gamma_h$ be deleting the vertices in $T$ and all incident edges.  Then $\Gamma_h[T]$ is the incomparability graph for a Hessenberg function $h[T]: [n-k]\to [n-k]$ as shown in \cite[Lemma 4.3]{HaradaPrecup2017}.   

\begin{remark}\label{rem: ht is dec}
 It is not difficult to see from the definitions of $ht(I_h)$ and $m(\Gamma_h)$ that $m(\Gamma_{h[T]})\leq m(\Gamma_{h})$, or equivalently, that $ht(I_{h[T]}) \leq ht(I_{h})$ (cf. also \cite[Proposition 5.8, Corollary 5.12, Lemma 5.13]{HaradaPrecup2017}). 
\end{remark}

Note that any acyclic orientation of $\Gamma_h$ induces an acyclic orientation of $\Gamma_h[T]$, as demonstrated in the example below.

\begin{example}\label{ex1} Let $h=(2,3,5,6,7,8,8,8)$, and consider the following acyclic orientation $\omega$ of $\Gamma_h$ displayed below.  
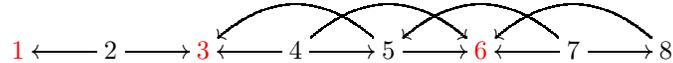
\vspace*{.15in}
\[\xymatrix{ {\color{red}1}  & {2}  \ar@[red][l]  \ar@[red][r]   & {\color{red}3}     & \ar@[red][l]  4 \ar[r]   \ar@[red]@/^1.5pc/[rr]    & {5} \ar@[red]@/_1.5pc/[ll] \ar@[red][r] & {\color{red}6} & 7  \ar@/_1.5pc/[ll]  \ar@[red][l] \ar[r] & 8 \ar@[red]@/_1.5pc/[ll]
}\]
This acyclic orientation has $T=\sk(\omega)=\{1, 3, 6 \}$, where the vertices in $\sk(\omega)$ and all incident edges are highlighted in red for emphasis.  For this graph, we have $m(\Gamma_h) = 3$.  The graph below shows $\Gamma[T]$ with the acyclic orientation induced from $\Gamma_h$.
\vspace*{.15in}
\[\xymatrix{  {2}     &  4 \ar[r]     & {5}  & 7\ar[r] \ar[l] & 8  
}  \]
which corresponds to the Hessenberg function $h[T] = (1, 3,4,5,5)$.  Note that we could also re-index the vertices of $\Gamma[T]$ to obtain the following acyclic graph.
\vspace*{.15in}
\[\xymatrix{  {1}     &  2 \ar[r]     & {3}  & 4\ar[r] \ar[l] & 5  
}  \]
\end{example}

An orientation $\omega\in \A(\Gamma_h)$ assigns each edge $e$ a source and a target; we notate the source (respectively target) of $e$ according to the orientation $\omega$ by $\src_\omega(e)$ (respectively $\tgt_\omega(e)$).  Given an orientation $\omega$ of $\Gamma_h$ we let
\[
\asc(\omega) := \{ e=\{a,b\} \hsm\vert\hsm \src_\omega(e) = a, \tgt_\omega(e)=b, \textup{ and } a<b\}.
\]
In other words, if $\Gamma_h$ is drawn as in Example~\ref{ex1} with the labels of the vertices increasing from left to right, then $\asc(\omega)$ counts the number of edges which point to the right.

Given a sink set $T\in \SK(\Gamma_h)$ the \textbf{degree of $T$} is 
\[
\deg_h(T):= \min\{\asc(\omega) \hsm\vert\hsm \omega\in \A(\Gamma_h) \textup{ and } \sk(\omega)=T\}.
\]
For example, $\deg_h(T) = 3$ for the $h$ and $T$ as appearing in Example~\ref{ex1}.  The next lemma is \cite[Lemma 4.8]{HaradaPrecup2017}, and shows that in practice it is easy to compute $\deg_h(T)$ for any $T\in \SK(\Gamma_h)$.

\begin{lemma}\label{lem: deg(T)} Let $T\in \SK(\Gamma_h)$.  Then
\[
\deg_h(T) = \lvert \{ e=\{a,b\}\in E(\Gamma_h) \hsm \vert \hsm a<b,\, b\in T\}\rvert.
\]
\end{lemma}  

We will see that sink sets in $\Gamma_h$ correspond bijectively to certain subsets of roots in $I_h$.  In particular, we need the following definition.

\begin{definition} Let $R\subseteq \Phi^-$.  We say $R$ is a \textbf{subset of height $k$} if there exist integers $q_1, q_2, \ldots, q_k, q_{k+1}\in [n]$ such that $q_1<q_2< \ldots < q_k<q_{k+1}$ and $R=\{t_{q_2}-t_{q_1}, t_{q_3}-t_{q_2}, \ldots, t_{q_{k+1}}-t_{q_k}\}$.  We let $\mathcal{R}_k(I)$ denote the set of all subsets of height $k$ in an ideal $I$, and define $\mathcal{R}(I):= \bigsqcup_{k\geq 0} \mathcal{R}_k(I)$.
\end{definition}

It is easy to show that $R\subseteq \Phi^-$ is a subset of height $k$ if and only if there exists $w\in \Symm_n$ such that $w(R)$ is a subset of simple roots corresponding to $k$ consecutive vertices in the Dynkin diagram for $\mathfrak{gl}(n,\C)$.  The set $\mathcal{R}(I)$ can also be used to compute the height of the ideal.  The following is \cite[Lemma 5.5]{HaradaPrecup2017}.

\begin{lemma}\label{lemma: height} Let $I$ be a nonempty ideal in $\Phi^-$.  Then $ht(I)=\max\{ |R| \hsm\vert\hsm R\in \mathcal{R}(I) \}$.
\end{lemma}

Recall that \cite[Section 5]{HaradaPrecup2017} defines a bijection:
\begin{eqnarray}\label{eqn: sink sets and subsets}
\SK_k(\Gamma_h)\to \mathcal{R}_{k-1}(I_h);\;\; T\mapsto R_T: = \{\beta_i = t_{\ell_{i+1}} - t_{\ell_i} \,\vert\, 1\leq i \leq k-1 \}
\end{eqnarray}
where $T=\{\ell_1< \ell_2< \cdots< \ell_k\}$.  By Lemma~\ref{lemma: height}, this bijection shows that the maximum size of any sink set in $\Gamma_h$ is precisely $ht(I_h)+1$, as noted in Section~\ref{sec:background}.

\begin{example} \label{ex2}  Let $h$ and $T$ be an in Example~\ref{ex1}.  The bijection defined in~\eqref{eqn: sink sets and subsets} above associates $T=\{1,3,6\}\in \SK_3(\Gamma_h)$ to the subset
\[
\{t_{3}-t_{1}, t_{6}-t_{3}\} \in \mathcal{R}_2(\Gamma_h).
\]
Since $3=m(\Gamma_h) = ht(\Gamma_h)+1$, we know that $I_h$ cannot contain any subsets of height $k\geq 3$.  This line of reasoning is essential for proving the inductive formulas later in this section.
\end{example}

%%%%%%%%%%%%%%%%%%%%%

\subsection{Another sink-set decomposition} 

Throughout this section, $\lambda=(\lambda_1, \lambda_2, \ldots, \lambda_k)$ is a partition of $n$ with $k$ parts.  In this section we will show that the sets $\W(\mathbb{J}_\lambda, h)$ have a sink set decomposition.  First we define a subset of $\W(\mathbb{J}_\lambda, h)$ associated to each sink set.

\begin{definition} \label{def: W-sink-set} Given $T=\{\ell_1< \ell_2< \cdots< \ell_k\}\in \SK_k(\Gamma_h)$ we define
\begin{eqnarray*}\label{eqn: W-sink-set}
\W_i (\mathbb{J}_{\lambda}, h, T) := \{w\in \W_i(\mathbb{J}_{\lambda}, h) \hsm\vert\hsm w(\ell_j) = k-j+1 ,\, 1\leq j \leq k \}.
\end{eqnarray*}
and let $\W(\mathbb{J}_\lambda, h, T)=\sqcup_i \W_{i}(\mathbb{J}_\lambda, h, T)$ where the union is taken over all $i$ such that $\W_i(\mathbb{J}_\lambda, h,T)\neq \emptyset$.
\end{definition}

The conditions defining $\W(\mathbb{J}_\lambda, h,T)$ tell us that if $w\in \W(\mathbb{J}_\lambda, h,T)$ then: 
\begin{eqnarray} \label{eqn: one-line condition}
k, k-1, \ldots, 2, 1 \textup{ appear in positions } \ell_1, \ell_2, \ldots, \ell_{k-1}, \ell_k \textup{ in the one-line notation for $w$.}  
\end{eqnarray}
In particular, $(k,k-1, \ldots, 2, 1)$ is a subsequence of the one-line notation for $w$.

\begin{example}\label{ex3} Let $h=(2,3,5,6,7,8,8,8)$ and $T=\{1,3,6\}$ as in Example~\ref{ex1}.  Consider $\lambda=(3,3,2)$; in this case $\mathbb{J}_\lambda = \{ \alpha_3, \alpha_6 \}$.
We have, for example, that $w\in \W(\mathbb{J}_\lambda, h,T)$ where
\[
w=[\mathbf{3},{6},\mathbf{2},8 ,5,\mathbf{1},7,4].
\]
\end{example}

Note that in the example above, $w^{-1}(\{\alpha_1, \alpha_2\}) = \{t_3-t_1, t_6-t_3\}=R_T$, where $R_T$ was computed in Example~\ref{ex2}.  The next lemma shows that this property characterizes the elements of $\W_i(\mathbb{J}, h, T)$.

\begin{lemma}\label{lem1} Let $T\in \SK_k(\Gamma_h)$.  Then  $w\in \W_i (\mathbb{J}_{\lambda}, h, T)$ if and only if $w\in \W_i(\mathbb{J}_\lambda, h)$ and $R_T=w^{-1}(\{\alpha_1, \ldots, \alpha_{k-1}\})$.
\end{lemma}
\begin{proof} If $w\in \W_i(\mathbb{J}_{\lambda}, h, T)$ for $T=\{\ell_1, \ell_2, \ldots, \ell_k\}$ then $w\in \W_i(\mathbb{J}_\lambda, h)$ and
\[
w^{-1}(\alpha_{k-j}) = w^{-1}(t_{k-j}-t_{k-j+1}) =  t_{\ell_{j+1}}-t_{\ell_j} \textup{ for all } j=1,\ldots, k-1
\]
by definition of $\W_i(\mathbb{J}_\lambda, h, T)$. 
Now the definition of $R_T$ given in~\eqref{eqn: sink sets and subsets} implies $w^{-1}(\{\alpha_1,\ldots, \alpha_{k-1}\}) = R_T\in \mathcal{R}_{k-1}(I_h)$, as desired. 

To show the converse, suppose  $w\in \W_i(\mathbb{J}_\lambda, h)$ and $w^{-1}(\{\alpha_1, \ldots, \alpha_{k-1}\}) = R_T$ where $T=\{\ell_1,\ell_2, \ldots, \ell_k\} \in \SK_k(\Gamma_h)$. Then
\[
w^{-1}(\{\alpha_1, \alpha_2, \ldots, \alpha_{k-1}\})=  R_T := \{t_{\ell_2}-t_{\ell_1}, t_{\ell_3}-t_{\ell_2}, \ldots, t_{\ell_k}-t_{\ell_{k-1}}\}.
\]
All that remains to show is that $w(\ell_j) = k-j+1$ for all $1\leq j \leq k$.  The equation above implies $w(\ell_j)\in \{1,2,\ldots, k\}$.  Observe that $w^{-1}(\{\alpha_1, \ldots, \alpha_{k-1}\}) = R_T$ implies $w(R_T) = \{\alpha_1, \ldots, \alpha_{k-1}\}$. Thus we also know $w(\ell_{j})=w(\ell_{j+1})+1$ since 
\[
w(t_{\ell_{j+1}} - t_{\ell_{j}}) = t_{w(\ell_{j+1})}-t_{w(\ell_{j})} \in \{\alpha_1, \ldots, \alpha_{k-1}\}. 
\]
This can only be the case if $\ell_1=k$, $\ell_2=k-1$, and so on.  We conclude $w(\ell_j)=k-j+1$, $1\leq j \leq k$ as desired.
\end{proof}

The next proposition generalizes the sink set decomposition given in~\eqref{eqn: sink set decomposition} and gives a sink set decomposition of the set $\W_i(\mathbb{J}_{\lambda}, h)$ for each $i$.  

\begin{proposition}\label{prop: prop1} 
Let $n$ be a positive integer and $h: [n] \to [n]$ a Hessenberg function.  Let $i\in \Z$, $i\geq 0$ and $\lambda$ be a partition of $n$ with $k$ parts.  Then 
\begin{eqnarray} \label{eqn: prop1-eq}
\W_i(\mathbb{J}_{\lambda}, h) = \bigsqcup_{T\in \SK_k(\Gamma_h)} \W_i(\mathbb{J}_{\lambda}, h, T).
\end{eqnarray}
\end{proposition}

We call the decomposition~\eqref{eqn: prop1-eq} the \textbf{sink set decomposition of $\W(\mathbb{J}_{\lambda}, h)$.}  

\begin{proof}  It is straightforward from the definition of the sets $\W_i (\mathbb{J}_\lambda, h,T)$ that the RHS of~\eqref{eqn: prop1-eq} is contained in the LHS.  Thus we have only to prove the opposite inclusion. Let $w\in \W_i(\mathbb{J}_\lambda, h)$.  By definition, $w^{-1}(\Delta\setminus\mathbb{J}_\lambda)\subseteq I_h$.  Since $\{\alpha_1, \ldots, \alpha_{k-1}\}\subseteq \Delta\setminus\mathbb{J}_{\lambda}$ it follows immediately that 
\[
t_{w^{-1}(1)}-t_{w^{-1}(2)}, t_{w^{-1}(2)}-t_{w^{-1}(3)},\ldots, t_{w^{-1}(k-1)}-t_{w^{-1}(k)}\in I_h.
\]
In particular, $R= w^{-1}(\{\alpha_1, \ldots, \alpha_{k-1}\})$ is a subset of $I_h$ of height $k-1$.  Since~\eqref{eqn: sink sets and subsets} is a bijection, there exists a unique sink set $T\in \SK_k(\Gamma_h)$ such that $R=R_T$ and therefore $w\in \W_i(\mathbb{J}_\lambda, h, T)$ by Lemma~\ref{lem1}.
\end{proof}

%%%%%%%%%%%%%%%%%%%%%%%

\subsection{Inductive Formulas} Our next goal is to identify each set $\W_i(\mathbb{J}, h, T)$ with a subset of permutations in $\Symm_{n-k}$.  The following notation generalizes \cite[Definition 7.3]{HaradaPrecup2017}.

\begin{definition}\label{def: shortest element} Suppose $T \in \SK_{k}(\Gamma_h)$ with $T= \{\ell_{1}< \ell_{2}< \cdots< \ell_{k}\}$ and $\lambda \vdash n$ with $k$ parts.  Define a permutation in $\Symm_n$, denoted $w_T$, by:
\begin{enumerate}
\item $w_T(\ell_j) = k-j+1$, $1\leq j\leq k$, i.e. $w_T$ satisfies~\eqref{eqn: one-line condition}, and
\item the remaining entries in the one-line notation of $w_T$ list the integers $[n]-T$ in increasing order from left to right.
\end{enumerate}
\end{definition}

\begin{example}\label{ex4} Let $h=(2,3,5,6,7,8,8,8)$ and $T=\{1,3,6\}$ as in Example~\ref{ex1}.  Then
\[
w_T=[\mathbf{3}, 4 ,\mathbf{2},5 ,6,\mathbf{1},7,8]
\]
where the entries in positions $\ell_1=1$, $\ell_2=3$ and $\ell_3=6$ are bolded for emphasis.  Note that $w_T$ need not be an element of $\W(\mathbb{J}_\lambda, h, T)$.  For example $w_T\notin \W(\mathbb{J}_\lambda, h, T)$ when $\lambda=(3,3,2)$ is the same partition considered in Example~\ref{ex3} since
\[
w_T^{-1}(\alpha_4) = w_{T}^{-1}(t_4-t_5) = t_2 - t_4\in \Phi_h 
\]
so $w_T$ does not satisfy the condition that $w_T^{-1}(\Delta \setminus \mathbb{J}_\lambda) \subseteq I_h$.
\end{example}

For each sink set $T=\{\ell_1< \ell_2< \cdots< \ell_k\}$ let $\mathsf{f}_T: ([n]\setminus T)\to [n-k]$ be the bijection such that $\phi_T(j)=j-j'$ where $j'$ denotes the number of elements $i\in T$ such that $i\leq j$. This bijection can be used to give explicit formulas for $w_T$, as noted in the following remark.

\begin{remark}\label{rem: explicit formula} 
The conditions defining $w_T$ can be written explicitly in formulas involving $\mathsf{f}_T$ as follows.
\begin{itemize}
\item If $j>k$ then $w_T^{-1}(j)$, the position of $j$ in the one-line notation for $w_T$, is the unique element of $[n]$ such that $\mathsf{f}_T(w_T^{-1}(j))=j-k$, and 
\item if $j\in [n]-T$ we have $w_T(j) = \mathsf{f}_T(j)+k$.  
\end{itemize}
\end{remark}

\begin{example}\label{ex5} Continuing Example~\ref{ex4} from above, we have $T=\{1,3,6\}$ and
\[
\mathsf{f}_T(2) =1, \, \mathsf{f}_T(4) = 2, \, \mathsf{f}_T(5) = 3,\,  \, \mathsf{f}_T(7)=4,\, \, \mathsf{f}_T(8)=5.
\]
Notice that $\mathsf{f}_T$ is the natural bijection we used to relabel the vertices of $\Gamma_h[T]$ in Example~\ref{ex1}.  The reader can easily verify the formulas given in Remark~\ref{rem: explicit formula} in this case.  For example,
\[
w_T(2) = \mathsf{f}_T(2) + 3 = 4 \;\textup{ and }\; \mathsf{f}_T(w_T^{-1}(6)) = \mathsf{f}_T(5) = 3 = 6-3. 
\]
\end{example}

The following is a generalization of \cite[Lemma 7.6]{HaradaPrecup2017}.

\begin{lemma}\label{lem: stabilizer}  Let $T=\{\ell_1< \ell_2< \cdots < \ell_k \}$ be a sink set of cardinality $k$.  Each element $w\in \Symm_n$ satisfying condition (1) of Definition~\ref{def: shortest element} can be written uniquely as $w=w_T\sigma$ where $\sigma\in \Stab(\ell_1, \ell_2,\ldots, \ell_{k})$.
\end{lemma}
\begin{proof} The hypotheses on $w$ determine the entries in positions $\ell_1, \ell_2, \ldots, \ell_{k}$ in one-line notation.  The other entries must be a permutation of the set $[n] \setminus \{\ell_1, \ell_2, \ldots, \ell_{k}\}$, and the hypotheses on $w$ place no conditions on this permutation.  Recall that for $w_T$ and any permutation $\sigma\in \Symm_n$, right-composition with $\sigma$ ``acts on the positions'', i.e. if $w_T$ sends $i$ to $w_T(i)$, then $w_T \sigma$ sends $i$ to $w_T(\sigma(i))$.  Thus, if $\sigma$ stabilizes $\ell_1, \ell_2, \ldots, \ell_{k}$, then $w=w_T\sigma $ satisfies $w(\ell_j) = w_T (\ell_j) =k-j+1$ for all $j=1,\ldots, k$.  Moreover, it is straightforward to see that such a $\sigma$ is unique.
\end{proof}

\begin{corollary}\label{cor: factorization} Each $w\in \W_i(\mathbb{J}_\lambda, h, T)$ can be written uniquely as $w=w_T\sigma$ where $\sigma \in \Stab(\ell_1, \ell_2, \ldots, \ell_k)$.
\end{corollary}
\begin{proof}  By definition, each element of $\W_i(\mathbb{J}_\lambda, h, T)$ satisfies condition (1) of Definition~\ref{def: shortest element}.
\end{proof}

\begin{example}\label{ex6} Let $w=[\mathbf{3},6,\mathbf{2},8,5,\mathbf{1},7,4] \in \W(\mathbb{J}_{(3,3,2)}, h, T)$ for $h=(2,3,5,6,7,8,8,8)$, as shown in Example~\ref{ex3}.    In this case, the factorization $w=w_T\sigma$ gives us
\[
\sigma=[\mathbf{1}, 5, \mathbf{3}, 8,4, \mathbf{6},7,2] \in \Stab(1,3,6).
\]
\end{example}

The bijection $\mathsf{f}_T$ defined above induces a natural isomorphism: 
\begin{eqnarray*}\label{eqn: Weyl group map}
 \Stab(\ell_1, \ell_2,\ldots, \ell_{k}) \to \Symm_{n-k}; \;\; \sigma \mapsto x_{\sigma}
\end{eqnarray*}
defined as follows.  Given $\sigma\in \Stab(\ell_1, \ell_2, \ldots, \ell_{rk})$, delete positions $\ell_1, \ell_2, \ldots, \ell_{k}$ from the one-line notation for $\sigma$ and then apply $\mathsf{f}_T$ to the remaining entries to obtain $x_{\sigma}$.  The result is clearly an element in $\Symm_{n-k}$ and each element of $\Symm_{n-k}$ arises in this way.  

\begin{example}\label{ex7} The element $\sigma = [\mathbf{1}, 5, \mathbf{3}, 8,4, \mathbf{6},7,2] \in \Stab(1,3,6)$ obtained in Example~\ref{ex6} above maps to $x_{\sigma} = [3, 5,2 ,4,1]\in \Symm_{5}$.
\end{example}

By Lemma~\ref{lem: stabilizer}, for each $T\in \SK_k(\Gamma_h)$ we get a well defined bijection
\[
\Psi_{T}: \{w\in \Symm_n : w \textup{ satisfies condition (1) of Definition~\ref{def: shortest element} }\} \to \Symm_{n-rk}
\]
defined by $\Psi_{T}(w_{\lambda,T}\sigma) = x_{\sigma}$.  Note that $\Psi_T$ is very similar to the map $d_{n,m}: \Symm_n \to \Symm_{n-m}$ defined in Section~\ref{sec:upper triangular} and used in the proof of Proposition~\ref{proposition: M-formulas Martha version}. Indeed, using the language of that section, applying $\Psi_T$ can be described as ignoring the $[k]$ entries in the one-line notation of $w$. 

Recall that there is a natural Lie subalgebra of $\mathfrak{gl}(n,\C)$ obtained by ``setting the variables in row/columns $\{\ell_1, \ell_2, \ldots, \ell_k\}$ equal to zero.''  More precisely, there is a natural Lie algebra isomorphism
\begin{eqnarray}\label{eqn: lie alg iso}
\{X\in \mathfrak{gl}(n,\C) \, \vert\, X_{ij}=0 \textup{ if } \{i,j\}\cap T \neq \emptyset\} \cong \mathfrak{gl}(n-k, \C).
\end{eqnarray}
defined explicitly on the basis $\{E_{ij}\, \vert \, \{i,j\} \cap T =\emptyset\}$ of the LHS by $E_{ij}\mapsto E_{\mathsf{f}_T(i)\mathsf{f}_T(j)}$.

Recall that for each $T\in \SK_k(\Gamma_h)$ we have an associated Hessenberg function $h[T]: [n-k] \to [n-k]$ whose incomparability graph is obtained by deleting the vertices in $T$ and any incident edges from $\Gamma_h$.  In fact, this Hessenberg function corresponds to the Hessenberg space $H\cap \mathfrak{gl}(n-k, \C)$ under the identification in~\eqref{eqn: lie alg iso}.  (See \cite[Section 4]{HaradaPrecup2017} for more details on this perspective.) We identify the set of roots
\[
\Phi[T]: = \{t_i-t_j \in \Phi \, \vert\, \{i,j\}\cap T = \emptyset \} \subseteq \Phi
\]
with the root system of $\mathfrak{gl}(n-k, \C)$ via 
\begin{eqnarray}\label{eqn: root system map}
t_i-t_j \mapsto t_{\mathsf{f}_T(i)}-t_{\mathsf{f}_T(j)}.
\end{eqnarray}

\begin{example}\label{ex7} We demonstrate the identifications from~\eqref{eqn: lie alg iso} and~\eqref{eqn: root system map} in the running example started in Example~\ref{ex1}, with $h=(2,3,5,6,7,8,8,8)$.  To visualize what is going on, we represent $\mathfrak{gl}(8,\C)$ as an $8\times 8$ grid with a star placed in the $(i,j)$-box precisely when the root $(i,j)$ is contained in $\Phi_h$.  The boxes highlighted in grey correspond the roots in $\Phi\setminus \Phi[T]$ so the white boxes containing a star correspond to the roots in $\Phi_h[T] :=\Phi[T]\cap \Phi_h$, to be discussed further below. 
\[\ytableausetup{centertableaux}
\mathfrak{gl}(8,\C): \;\; \begin{ytableau} 
*(grey)\star & *(grey)\star & *(grey)\star & *(grey)\star & *(grey)\star & *(grey)\star &*(grey)\star & *(grey)\star \\ 
*(grey)\star & \star & *(grey)\star & \star & \star & *(grey)\star & \star & \star \\ 
*(grey)\empty & *(grey)\star & *(grey)\star & *(grey)\star & *(grey)\star & *(grey)\star & *(grey)\star & *(grey)\star \\ 
*(grey)\empty & \empty & *(grey)\star & \star & \star & *(grey)\star & \star & \star \\  
*(grey)\empty & \empty & *(grey)\star & \star &  \star & *(grey)\star &  \star &  \star\\ 
*(grey)\empty & *(grey)\empty & *(grey)\empty & *(grey)\star & *(grey)\star & *(grey)\star & *(grey)\star & *(grey)\star \\ 
*(grey)\empty & \empty & *(grey)\empty & \empty & \star & *(grey)\star & \star & \star \\ 
*(grey)\empty & \empty & *(grey)\empty & \empty & \empty & *(grey)\star & \star & \star \\    
\end{ytableau} \quad \quad\quad
\mathfrak{gl}(5, \C): \;\; \begin{ytableau} 
\star & \star & \star & \star & \star  \\  
\empty & \star & \star & \star & \star \\ 
\empty  & \star & \star & \star & \star \\ 
\empty & \empty & \star & \star & \star  \\ 
\empty & \empty & \empty & \star & \star \\   
\end{ytableau} 
\]
\end{example}

Note that the map in~\eqref{eqn: root system map} is an isomorphism of root systems, where $\Phi[T]$ is viewed as a subroot system of $\Phi$ (since $\Phi[T]$ is closed under addition in $\Phi$).  Moreover, the subsets $\Phi_h[T]:= \Phi_h \cap \Phi[T]$ and $\Phi_h^-[T]:= \Phi_h^-\cap \Phi[T]$ correspond to $\Phi_{h[T]}$ and $\Phi_{h[T]}^-$ respectively, via~\eqref{eqn: root system map}.

\begin{remark}\label{rem: compatibility} The root system isomorphism given in~\eqref{eqn: root system map} is compatible with the corresponding identification $\Stab(\ell_1, \ldots, \ell_k)$ given in~\eqref{eqn: Weyl group map}.  Indeed, if $\sigma\in \Stab(\ell_1, \ldots, \ell_k)$ and $t_i-t_j\in \Phi[T]$ then $\sigma(t_i-t_j)\in \Phi[T]$ and 
\[
t_k-t_{\ell} = \sigma(t_i-t_j) \Leftrightarrow t_{\mathsf{f}_T(k)}-t_{\mathsf{f}_T(\ell)} = x_{\sigma}(t_{\mathsf{f}_T(i)}- t_{\mathsf{f}_T(j)}).
\]
Recall that for a permutation $w \in \Symm_n$ we define 
\[
\inv(w) := \{ (i,j) \, \mid \, i > j \textup{ and } w(i) < w(j) \}.
\]
Then~\eqref{eqn: root system map} gives a bijection between $\inv(\sigma)\cap \Phi[T]$ and $\inv(x_{\sigma})$ and a bijection between $\inv_h(\sigma)\cap \Phi[T]$ and $\inv_{h[T]} (x_{\sigma})$.
\end{remark}

\begin{lemma}\label{lemma: inversions} Let $T\in \SK_k (\Gamma_h)$. Then 
\begin{enumerate}
\item $\inv(w_{T})=\{ (i,j) \, \vert\, i>j \textup{ and } i\in T \}$, and
\item   if $w=w_T\sigma$ for $\sigma\in \Stab(\ell_1, \ldots, \ell_k)$ then
\begin{eqnarray}\label{eqn: coset decomp}
\inv(w ) = \inv(w_T) \sqcup (\inv(\sigma) \cap \Phi[T]).
\end{eqnarray}
\end{enumerate}
\end{lemma}

\begin{proof}

We begin by proving statement (1).  If $(i,j)\in \inv(w_T)$ then $i>j$ and $w_T(i)<w_T(j)$.   If $i\notin T$, then $w_T(j)>w_T(i)>k$ so from the construction of $w_T$ we conclude $j\notin T$.  But the entries in the one-line notation of $w_T$ for $i, j \not \in T$ cannot be inverted, by Definition~\ref{def: shortest element}(2). Hence $w_T(j) > w_T(i)$, yielding a contradiction.  
Therefore $i\in T$ as desired. On the other hand, consider $(i,j)$ with $i>j$ and $i\in T$.  Since $i\in T$, we may write $i=\ell_{i_0}$ for some $i_0$ with $1\leq i_0\leq k$.  If $j\in T$, then $j=\ell_{j_0}$ for some $j_0$ with $1\leq j_0\leq k$ such that $j_0<i_0$ (since $j<i$) and we have
\[
w_T(i) = w_T(\ell_{i_0}) = k-i_0+1 < k-j_0+1 =w_T(\ell_{j_0})=w_T(j)
\] 
so $(i,j)\in \inv(w_T)$.  If $j\notin T$, then $w_T(j)>k$ and therefore
\[
w_T(i)\leq k < w_T(j)
\]
so $(i,j)\in \inv(w_T)$ also. This proves (1).

Next we prove (2). Let $w$ be as given. Note that since $\sigma\in \Stab(\ell_1, \ldots, \ell_k)$, we have $w(T) = w_T(T)=\{1,2,\ldots, k\}$.  Our proof relies on this fact, as well as the formulas given in Remark~\ref{rem: explicit formula}. We first show the inclusion $\inv(w) \subseteq \inv(w_T) \sqcup (\inv(\sigma) \cap \Phi[T])$.  Let $(i,j)\in \inv(w)$.  If $i\in T$ then $(i,j)\in \inv(w_T)$ by (1).  If $i \not \in T$, then $k< w(i)<w(j)$ so $j\notin T$ as above and we conclude $(i,j)\in \Phi[T]$.  Since $\sigma \in \Stab(\ell_1, \ldots, \ell_k)$ and $\sigma$ is a permutation $\sigma$ also preserves the complement $[n] \setminus \{\ell_1, \ldots, \ell_k\} = [n] - T$. Hence if $i \not \in T$ then $\sigma(i) \not \in T$ also. Using this fact and the formulas from Remark~\ref{rem: explicit formula} we now have
\[
\mathsf{f}_T(\sigma(i))+k = w_T\sigma(i)< w_T\sigma(j) = \mathsf{f}_T(\sigma(j))+k \Rightarrow \mathsf{f}_T(\sigma(i))< \mathsf{f}_T(\sigma(j)) \Rightarrow  \sigma(i)<\sigma(j)
\]
since $\mathsf{f}_T^{-1}$ is an increasing function. Therefore $(i,j)\in \inv(\sigma)\cap \Phi[T]$.  

To prove the opposite inclusion, suppose $(i,j)\in \inv(w_T)$.  By (1), we know $i\in T$. If $j\in T$ then 
\[
w(i) = w_T\sigma(i)=w_T(i)< w_T(j) = w_T\sigma(j)=w(j)
\]
so $(i,j)\in \inv(w)$.  If $j\notin T$ then $w(j)=w_T\sigma(j)>k$ and
\[
w(i) = w_T\sigma(i) = w_T(i)\leq k <  w(j)
\]
so $(i,j)\in \inv(w)$ in this case also.  Hence $\inv(w_T) \subseteq \inv(w)$. Next suppose $(i,j)\in \inv(\sigma)\cap \Phi[T]$.  This means $i,j\notin T$ and thus we know, as above, that $\sigma(i), \sigma(j)\notin T$ also. Hence
\[
w(i) = w_T\sigma(i)=f_T(\sigma(i))+k < f_T(\sigma(j))+k = w_T\sigma(j)=w(j)
\]
since $f_T$ is increasing and $\sigma(i)< \sigma(j)$ by assumption. Therefore $\inv(\sigma) \cap \Phi[T] \subseteq \inv(w)$ also. This completes the proof.
\end{proof}

\begin{example}\label{ex8} Continuing the running example, we have
\[
w_T\sigma=w=[\mathbf{3},6,\mathbf{2},8,5,\mathbf{1},7,4] \in \W(\mathbb{J}_{(3,3,2)}, h, T)
\]
where $w_T= [\mathbf{3}, 4 ,\mathbf{2},5 ,6,\mathbf{1},7,8]$ and $\sigma=[\mathbf{1}, 5, \mathbf{3}, 8,4, \mathbf{6},7,2] \in \Stab(1,3,6)$. In this case it can be checked that 
\begin{eqnarray*}
\inv(w) = \{ (6,1), (6,2), (6,3), (6,4), (6,5), (3,1), (3,2), (8,2), (8,4), (8,5), (8,7), (5,2), (5,4), (7,4)  \}  
\end{eqnarray*}
where
\[
\inv(w_T) = \{ (6,1), (6,2), (6,3), (6,4), (6,5), (3,1), (3,2) \}
\]
and
\[
\inv(\sigma)\cap \Phi[T] = \{ (8,2), (8,4), (8,5), (8,7), (5,2), (5,4), (7,4) \}.
\]

\end{example}

It is, in general, not the case that $\ell(w)=\ell(w_T)+\ell(\sigma)$ (where $\ell(w)$ denotes the Bruhat length of $w\in \Symm_n$); indeed, this is not true for the example above.  Therefore the decomposition of the inversions given in Lemma~\ref{lemma: inversions} above is not a simple application of known formulas for the inversion set of a given permutation.

\begin{lemma}\label{lemma: induction for the J} 
Let $\lambda=(\lambda_1, \lambda_2, \ldots, \lambda_k)$ be a partition of $n$ with exactly $k$ parts and $T \in \SK_k(\Gamma_h)$.  Then:
\begin{enumerate}
\item  $w_{T }^{-1}(\mathbb{J}_\lambda)\cap \Phi[T]$ is mapped to $\mathbb{J}_{\lambda[1]}$ under the identification in~\eqref{eqn: root system map} and 
\item $w_T^{-1}(\Delta \setminus \mathbb{J}_\lambda) \cap \Phi[T]$ is mapped to $\{\alpha_1, \ldots, \alpha_{n-k-1}\} \setminus \mathbb{J}_{\lambda[1]}$ under the identification in~\eqref{eqn: root system map},
\end{enumerate}
where $\lambda[1]=(\lambda_1-1, \lambda_2-1, \ldots, \lambda_k-1)$.
\end{lemma}

\begin{proof} 
By definition, $w_T(T) = \{1, 2, \ldots, k\}$.  Therefore 
\begin{eqnarray*}
w_T^{-1}(\alpha_j) = t_{w_T^{-1}(j)}-t_{w_T^{-1}(j+1)} \in \Phi[T] &\Leftrightarrow& \{w_T^{-1}(j), w_T^{-1}(j+1)\}\cap T=\emptyset\\ &\Leftrightarrow& \{j,j+1\}\cap \{1,2,\ldots, k\}=\emptyset
\end{eqnarray*}
and we conclude that $w_T^{-1}(\alpha_j)\in \Phi[T]$ if and only if $k+1\leq j \leq n-1$. Let $\mathbb{J}_{\lambda[1]+k} := \{\alpha_{i+k} : \alpha_i \in \mathbb{J}_{\lambda[1]}\}$ and $\mathbb{J}_{\lambda[1]+k}^c := \{\alpha_{k+1}, \ldots, \alpha_{n-1}\}\setminus \mathbb{J}_{\lambda[1]+k}$.  By definition,
\[
\mathbb{J}_\lambda = \{\alpha_k\} \sqcup \mathbb{J}_{\lambda[1]+k} \textup{ and  } \Delta\setminus\mathbb{J}_\lambda= \{\alpha_1, \ldots, \alpha_k \}\sqcup \mathbb{J}_{\lambda[1]+k}^c.
\]
Thus, $w_T^{-1}(\mathbb{J}_\lambda)\cap \Phi[T] = w_T^{-1}(\mathbb{J}_{\lambda[1]+k})$ and $w_{T}^{-1}(\Delta \setminus \mathbb{J}_\lambda)\cap \Phi[T] = w_T^{-1}(\mathbb{J}_{\lambda[1]+k}^c)$. Suppose $j>k$.  From the formula given in Remark~\ref{rem: explicit formula} we have
\[
w_T^{-1}(\alpha_j) = t_{w_T^{-1}(j)} - t_{w_T^{-1}(j+1)} \mapsto t_{\mathsf{f}_T(w_T^{-1}(j))} - t_{\mathsf{f}_T(w_T^{-1}(j+1))} = t_{j-k}-t_{j+1-k} 
\]
under the identification in~\eqref{eqn: root system map}.  Therefore~\eqref{eqn: root system map} maps $w_T^{-1}(\mathbb{J}_{\lambda[1]+k})$ to $\mathbb{J}_{\lambda[1]}$ and $w_T^{-1}(\mathbb{J}_{\lambda[1]+k}^c)$ to $\{\alpha_1,\ldots, \alpha_{n-k-1}\}\setminus\mathbb{J}_{\lambda[1]}$ as desired.
\end{proof}

The next lemma is the technical heart of our argument. Notice that this is the first time we require the assumption that $k=ht(I_h)+1$.

\begin{lemma} \label{lemma: Hessenberg condition}  
Let $\lambda=(\lambda_1, \lambda_2, \ldots, \lambda_k)$ be a partition of $n$ with $k$ parts, where $k=ht(I_h)+1$ and $T \in \SK_k(\Gamma_h)$. Then $w=w_T\sigma\in \W(\mathbb{J}_\lambda, h, T)$ if and only if $\Psi_T(w) = x_\sigma \in \W(\mathbb{J}_{\lambda[1]}, h[T])$.
\end{lemma}

\begin{proof}  By Corollary~\ref{cor: factorization}, each $w\in \W(\mathbb{J}_{\lambda}, h, T)$ is of the form $w=w_T\sigma$ for a unique $\sigma\in \Stab(\ell_1, \ell_2, \ldots, \ell_k)$ and 
\[
w^{-1}(\mathbb{J}_{\lambda})\subseteq \Phi_h \; \textup{ and } \; w^{-1}(\Delta\setminus\mathbb{J}_{\lambda}) \subseteq I_h.
\]
Since $\Phi[T]$ is invariant under $\sigma$ and $\Phi_h[T]=\Phi[T]\cap \Phi_h$, intersecting the sets appearing in the equations above with $\Phi[T]$ yields
\begin{eqnarray}\label{eq1}
\sigma^{-1}(w_T^{-1}(\mathbb{J}_\lambda)\cap \Phi[T]) \subseteq \Phi_h[T] \textup{ and } \sigma^{-1}(w_T^{-1}(\Delta\setminus\mathbb{J}_{\lambda}) \cap \Phi[T])\subseteq I_h[T]
\end{eqnarray}
 where $I_h[T] := \Phi[T] \cap I_h$.
The forward direction of the statement now follows directly from Lemma~\ref{lemma: induction for the J} and Remark~\ref{rem: compatibility}.

On the other hand, if $x_{\sigma}\in \W(\mathbb{J}_{[1]}, h[T])$ then Lemma~\ref{lemma: induction for the J} and Remark~\ref{rem: compatibility} together imply that equation~\eqref{eq1} still holds.  In order to show $w=w_T\sigma\in \W(\mathbb{J}_\lambda, h, T)$ we must prove $w^{-1}(\alpha_k)\in \Phi_h$ and $w^{-1}(\{\alpha_1, \cdots, \alpha_{k-1}\}) \subseteq I_h$.  The latter fact is straightforward, since from the definition of $w_T$ we have
\[
w^{-1}(\{\alpha_1, \ldots, \alpha_{k-1}\})= R_T\subseteq I_h.
\]
Thus, we have only to show that $w^{-1}(\alpha_k)\in \Phi_h$.  If not, then $w^{-1}(\alpha_k)\in I_h$ and 
\[
R = w^{-1}(\{\alpha_1, \ldots, \alpha_{k-1},\alpha_k\}) \subseteq I_h
\]
is a subset of height $k$ in $I_h$.  Lemma~\ref{lemma: height} now implies $ht(I_h)>k-1$, a contradiction.  We conclude that $w\in \W(\mathbb{J}_\lambda, h,T)$ as desired.
\end{proof}

We are now ready to prove the main result of this section.

\begin{theorem}\label{thm: inductive formula}  Let $\lambda$ be a partition of $n$ with $k$ parts, where $k=ht(I_h)+1$ and $T \in \SK_k(\Gamma_h)$.  Then $\Psi_{T}$ maps $\W_i(\mathbb{J}_{\lambda}, h ,  T)$ bijectively onto $\W_{i-\deg_h(T)}(\mathbb{J}_{\lambda[1]}, h[T])$.
\end{theorem}

\begin{proof}  Let $w\in \W_i(\mathbb{J}_\lambda, h, T)$ and $T=\{\ell_1< \ell_2< \cdots< \ell_k\}$.  By Corollary~\ref{cor: factorization}, $w=w_T\sigma$ for a unique $\sigma\in \Stab(\ell_1, \ldots, \ell_k)$ and $\Psi_T(w) = x_{\sigma}$ by definition.    Lemma~\ref{lemma: Hessenberg condition} implies $\Psi_T: \W(\mathbb{J}_\lambda, h, T) \to \W(\mathbb{J}_{\lambda[1]}, h[T])$ is a bijection, so we have only to show that this bijection respects the grading as indicated.  But this follows from Lemma~\ref{lemma: inversions} by intersecting both sides of~\eqref{eqn: coset decomp} with $\Phi_h$. We obtain
\begin{eqnarray*}
\inv_h(w)  = \inv_h(w_T) \sqcup (\inv(\sigma)\cap \Phi_h[T]) 
\end{eqnarray*}
so
\[
i =|\inv_h(w)| =  |\inv_h(w_T)| + |\inv_{h[T]}(x_{\sigma})| = \deg_h(T) +  |\inv_{h[T]}(x_{\sigma})|
\]
where the equation above follows directly from Lemma~\ref{lem: deg(T)} and Remark~\ref{rem: compatibility}.  From this it follows that $\Psi_T(w)\in \W_{i-\deg_h[T]}(\mathbb{J}_{\lambda[1]}, h[T])$ as desired.
\end{proof}

%%%%%%%%%%%%%%%%%%%%%
%  Formula
%%%%%%%%%%%%%%%%%%%%%

\section{Inductive formulas for the multiplicities associated to maximal sink sets}\label{sec:closed formula}

The main result of this section is  a first application of the results obtained in the previous sections. Specifically, we derive an inductive formula for the 
 multiplicities $c_{\mu,i}$ of the tabloid representations in the decomposition of the dot action representation on $H^{2i}(\Hess(\mathsf{S},h))$, 
 for partitions $\mu$ with the maximal number of parts. This result proves \cite[Conjecture 8.1]{HaradaPrecup2017}. 
%
%where the decomposition is with respect to the basis of tabloid representations. In equation~\eqref{eq: decomp into Mlambda} these coefficients are notated by $c_{\mu,i}$, where $\mu$ varies over the partitions of $n$. 

In the following we use the notation and terminology of Section~\ref{sec: inductive formula}. Let $n$ be a positive integer, $h: [n] \to [n]$ a Hessenberg function, $\Gamma_h$ its associated incomparability graph. Let $k=ht(I_h)+1$. Let $\omega \in \A_k(\Gamma_h)$ be an acyclic orientation of $\Gamma_h$ and let $T = \sk(\omega)$ be the sink set of $\omega$ of maximal size $k$. We can delete the vertices of $T$ and all incident edges from $\Gamma_h$ to obtain a strictly smaller graph $\Gamma_{h[T]}$ associated to a smaller Hessenberg function $h[T]: [n-k] \to [n-k]$ (see \cite[Section 4]{HaradaPrecup2017} for more details). 

Let $\mathsf{S}_{n-k}\in \mathfrak{gl}(n-k, \C)$ be a regular semisimple operator.  The cohomology of the Hessenberg variety $\Hess(\mathsf{S}_{n-k}, h[T]) \subseteq \Flags(\C^{n-k})$ has a dot action of the permutation group $\Symm_{n-k}$ and therefore has a corresponding decomposition analogous to~\eqref{eq: decomp into Mlambda}. We denote the coefficients for this decomposition by $c_{\mu', i}^T$ as follows: 
\begin{equation}\label{eq: decomp into Mlambda small}
H^{2i}(\Hess(\mathsf{S}_{n-k}, h[T])) = \sum_{\mu' \vdash (n-k)} c_{\mu', i}^T M^{\mu'}.
\end{equation}

With the notation in place we can state our inductive formula, which was first stated as Conjecture 8.1 in \cite{HaradaPrecup2017}.

\begin{theorem}\label{theorem: max sink set case} 
Let $n$ be a positive integer and $h: [n] \to [n]$ a Hessenberg function. 
Let $k=ht(I_h)+1$. Suppose $\mu\vdash n$ is a partition of $n$ with exactly $k=ht(I_h)+1$ parts.  Then for all $i\geq 0$ we have 
\begin{equation}\label{eq:formula for c}
c_{\mu, i} = \sum_{T \in\SK_k(\Gamma_h)} c_{\mu[1], i-\deg_h(T)}^{T}.
\end{equation}
\end{theorem}

\begin{proof} 

Let $\Par_{\geq k}(n)$ denote the set of all partitions of $n$ with at least $k$ parts and $\Par_k(n)$ denote the set of all partitions of $n$ with exactly $k$ parts.  Let $\overline{A}=(A(\lambda, \mu))_{\lambda, \mu\in \Par_{\geq k}(n)}$.  By definition, if $\lambda \in \Par_{\geq k}(n)$ and $\lambda \preceq \mu$, then $\mu$ has at least $k$ parts so $\overline{A}$ is the lower right-hand $|\Par_{\geq k}(n)|\times |\Par_{\geq k}(n)|$ submatrix of $A$.  In particular, $\overline{A}$ is upper-triangular since $A$ is by Theorem~\ref{theorem: upper-triangular}. We consider the matrix equation
\begin{eqnarray}\label{eqn: matrix eqn2}
\overline{A} \, \overline{X}_i = \overline{W}_i \; \textup{ where } \; \overline{X}_i = (c_{\mu, i})_{\mu\in \Par_{\geq k}(n)} \; \textup{ and } \; \overline{W}_i = (|\W_i(\mathbb{J}_\lambda, h)|)_{\lambda\in \Par_{\geq k}(n)}.
\end{eqnarray}
The matrix equation appearing in~\eqref{eqn: matrix eqn2} is consistent since we already know a priori that there exists a solution, given by the coefficients $c_{\mu,i}$ of~\eqref{eq: decomp into Mlambda}. Moreover, since $\overline{A}$ is upper-triangular, this solution is unique.  Furthermore, $c_{\mu, i}=0$ for all partitions $\mu$ with more than $k$ parts by Theorem~\ref{lemma: vanishing coefficients}.  We may therefore rewrite the matrix equation $\overline{A}\, \overline{X}_i=\overline{W}_i$ as the following system of linear equations, one equation for each partition $\lambda\vdash n$ with exactly $k$ parts:
\begin{eqnarray}\label{eqn: linear relations2}
|\W_i(\mathbb{J}_\lambda, h)| = \sum_{\mu\in \Par_k(n)} c_{\mu, i} A(\lambda, \mu).
\end{eqnarray}

In order to proved the desired result, it suffices to show that the RHS of~\eqref{eq:formula for c} satisfies, as $\mu$ varies among all partitions of $n$ with exactly $k$ parts, the linear relations obtained in~\eqref{eqn: linear relations2}.  From the sink set decomposition of $\W_i(\mathbb{J}_\lambda, h)$ given in Proposition~\ref{prop: prop1} and the bijection between $\W_i(\mathbb{J}_\lambda, h, T)$ and $\W_{i-\deg_h(T)}(\mathbb{J}_{\lambda[1]}, h[T])$ given in Theorem~\ref{thm: inductive formula} we obtain 
\begin{eqnarray*} 
\lvert \W_i(\mathbb{J}_{\lambda}, h)\rvert  &=& \sum_{T \in \SK_{k}(\Gamma_h)} \lvert \W_{i - \deg_h(T)}(\mathbb{J}_{\lambda[1]}, h[T])\rvert  \\ 
&=& \sum_{T \in \SK_k(\Gamma_h)} \sum_{\mu' \vdash (n-k)} c_{\mu', i - \deg_h(T)}^T \; \lvert \mathcal{D}(\mathbb{J}_{\lambda[1]}, J_{\mu'}) \rvert 
\end{eqnarray*}
where the second equality follows from Theorem~\ref{theorem: linear relations}, applied to $J=\mathbb{J}_{\lambda[1]}, h[T]$ and $n-k$.  Notice that $\lvert \mathcal{D}(\mathbb{J}_{\lambda[1]}, J_{\mu'}) \rvert  = A(\lambda[1], \mu')$ by~\eqref{eq: def A lambda mu}.

From Remark~\ref{rem: ht is dec} it follows that for any $T \in \SK_{k}(\Gamma_h)$, the height of the ideal $I_{h[T]}$ is at most $k-1=ht(I_h)$ and hence the coefficient $c_{\mu', i-\deg_h(T)}^{T}$ appearing in the last expression above is zero if $\mu'$ has more than $k$ parts. Therefore we may rewrite the above expression and exchange the summation operations as follows:  
\begin{equation*}
\begin{split} 
\sum_{T\in \SK_{k}(\Gamma_h)} \sum_{\mu' \vdash (n-k)} c_{\mu', i - \deg_h(T)}^T\; A(\lambda[1], \mu') & = 
\sum_{T\in \SK_{k}(\Gamma_h)} \sum_{\substack{\mu' \vdash (n-k) \\ \mu' \textup{ has at most $k$ parts}}} c_{\mu', i - \deg_h(T)}^T\; A(\lambda[1], \mu') \\
& = \sum_{\substack{\mu' \vdash (n-k) \\ \mu' \textup{ has at most $k$ parts}}} \left( \sum_{T\in \SK_{k}(\Gamma_h)} c_{\mu', i - \deg_h({T})}^{T} \right)A(\lambda[1], \mu'). 
\end{split} 
\end{equation*} 
Next we observe that any partition $\mu'$ of $n-k$ which has at most $k$ parts is equal to $\mu[1]$ for a unique partition $\mu$ of $n$ with the properties that $\mu$ has exactly $k$ parts. Indeed, it is not hard to see that $\mu:=(\mu'_1 + 1, \mu'_2+1, \ldots, \mu'_k+1)$ is precisely this (unique) $\mu$.

Using this correspondence $\mu \leftrightarrow \mu[1]=\mu'$, we may 
therefore conclude that the last expression in the displayed equations above is equal to
\[
\sum_{\mu\in \Par_k(n)} \left( \sum_{T \in \SK_{k}(\Gamma_h)} c_{\mu[1], i - \deg_h({T})}^{T} \right)A(\lambda[1], \mu[1])
\]
which is in turn equal to 
\[
\sum_{\mu\in \Par_k(n)} \left( \sum_{T\in \SK_{k}(\Gamma_h)} c_{\mu[1], i - \deg_h({T})}^{T} \right) A(\lambda, \mu)
\]
by Corollary~\ref{corollary: truncate A}. Putting the above together we have obtained
\begin{equation}
\lvert \W_i(\mathbb{J}_{\lambda}, h)\rvert  = 
\sum_{\mu\in \Par_k(n)} \left( \sum_{{T}\in \SK_{k}(\Gamma_h)} c_{\mu[1], i - \deg_h({T})}^{T} \right) A(\lambda, \mu).
\end{equation}
This proves the desired result.
\end{proof}

%%%%%%%%%%%%%%%%%%%%%
%  Bibliography
%%%%%%%%%%%%%%%%%%%%%

\def\cprime{$'$}

\end{document}